\documentclass[10pt]{article}

\usepackage{amsmath,amsxtra,amssymb,latexsym, amscd,amsthm}
\usepackage{here}
\usepackage{graphpap}
\usepackage{graphicx}
\usepackage{color}
\usepackage{hyperref} 
\usepackage{cases}
\usepackage{todonotes}
\usepackage[linesnumbered,ruled]{algorithm2e}

\newtheorem{theorem}{Theorem}[section]

\newtheorem{lemma}[theorem]{Lemma}

\theoremstyle{definition}

\newtheorem{remark}[theorem]{Remark}

\numberwithin{equation}{section}

\def\divv{\text{div}}
\def\<{\langle}
\def\>{\rangle}

\def\N{\mathbb{N}}

\def\limn{\lim_{n\to\infty}}
\def\limsupn{\limsup_{n\to\infty}}
\def\liminfn{\liminf_{n\to\infty}}
\def\limm{\lim_{m\to\infty}}
\def\div{\mbox{div}}

\definecolor{purple}{rgb}{0.4, 0.0, 0.4}

\begin{document}


\begin{center}
\large \bf Finite element approximation of source term identification with TV-regularization
\end{center}

\vspace{0.3cm}

\centerline {\bf Michael Hinze$^a$\let\thefootnote\relax\footnote{Email: michael.hinze@uni-hamburg.de ~$\cdot$~ quyen.tran@uni-goettingen.de} ~$\cdot$~ Tran Nhan Tam Quyen$^{b,*}$\footnote{$^*$Corresponding author}}

{\small $^a$Department of Mathematics, University of Hamburg, Bundesstr. 55, 20146 Hamburg, Germany\\
$^b$Institute for Numerical and Applied Mathematics, University of Goettingen, Lotzestr. 16-18, 37083 Goettingen, Germany
}

\vspace{0.5cm}

{\small {\bf Abstract:}
In this paper we investigate the problem of recovering the source $f$ in the elliptic system
\begin{equation*}
\begin{aligned}
-\nabla \cdot \big(\alpha \nabla u \big) + \beta u &= f \mbox{~in~} \Omega,  \\
\alpha \nabla u \cdot \vec{n} +\sigma u &= j \mbox{~on~} \partial\Omega  
\end{aligned}
\end{equation*}
from an observation $z$ of the state $u$ on a part $\Gamma$ of the boundary $\partial\Omega$, where the functionals $\alpha,\beta,\sigma$ and $j$ are given. For the particular interest in reconstructing probably discontinuous sources, we use the standard least squares method with the total variation regularization, i.e. we consider a minimizer of the minimization problem
$$
\min_{f\in F_{ad}} J(f), \quad J(f) := \frac{1}{2}\|u(f)-z\|^2_{\Gamma} + \rho TV(f) \eqno \left(\mathcal{P}\right)
$$
as reconstruction. Here $u(f)$ denotes the unique weak solution of the above elliptic system which depends on the source term $f$, $TV(f)$ is the total variation of $f$, $\rho>0$ is the regularization parameter, the admissible set $F_{ad} := \left\{f\in BV(\Omega) ~|~ -\infty<\underline{f} \le f(x) \le \overline{f} <\infty \quad\mbox{for a.e. in} \quad \Omega\right\}$
with $\underline{f}$, $\overline{f}$ being given constants, and $BV(\Omega)$ is the Banach space of all bounded total variation functionals. We approximate problem $\left(\mathcal{P}\right)$ with piecewise linear and continuous finite elements, where $u^h \in \mathcal{V}^h_1$ denotes the corresponding finite element approximation of $u$. This leads to the minimization problem
$$
\min_{f\in F^h_{ad}} J^h(f), \quad J^h(f) := \frac{1}{2}\|u^h(f)-z\|^2_{\Gamma} + \rho TV(f), \eqno \big(\mathcal{P}^h\big)
$$
where $F^h_{ad} = F_{ad} \cap \mathcal{V}^h_1$. We in Theorems \ref{odinh1}, \ref{stability2} provide the numerical analysis for the discrete solutions $f^h$ of $\big(\mathcal{P}^h\big)$, and also propose an algorithm to stably solve this discrete minimization problem, where we are lead by the algorithmic developments of \cite{bartels-12,tian-yuan-16}. In particular we prove that the iteration sequence $\big(f^h_n\big)_n$ generated by this algorithm converges to a minimizer of $\big(\mathcal{P}^h\big)$, and that convergence measure of the kind
\begin{align*}
\|f^h_{n+1}-f^h_n\|^2 =\mathcal{O}\left( \frac{1}{n}\right)
\end{align*}
is satisfied. Finally, a numerical experiment is presented to illustrate our theoretical findings.
}



\section{Introduction}

Let $\Omega$ be an open bounded connected set of $\mathbb{R}^d,
~d\ge 2$ with the polygonal boundary $\partial \Omega$. In this paper we investigate the inverse problem of identifying the {\it source term} $f \in L^2(\Omega)$ in the elliptic system 
\begin{equation}\label{17-5-16ct1}
\begin{aligned}
-\nabla \cdot \big(\alpha \nabla u \big) + \beta u &= f \quad\mbox{in}\quad \Omega,  \\
\alpha \nabla u \cdot \vec{n} +\sigma u &= j \quad\mbox{on}\quad \partial\Omega  
\end{aligned}
\end{equation}
from a boundary measurement $z\in L^2(\Gamma)$ of the exact data $g^\dag := u_{|\Gamma}$, where $\vec{n}$ is the unit outward normal on $\partial\Omega$ and $\Gamma\subset\partial\Omega$ is a relatively open subset of the boundary. 

In the system \eqref{17-5-16ct1} the Robin boundary condition $j\in H^{-1/2}(\partial\Omega) := \big(H^{1/2}(\partial\Omega)\big)^*$ and special functionals $\alpha, \beta, \sigma$ are assumed to be given, where $\sigma \in L^\infty(\partial\Omega)$ with $\sigma(x)\ge 0$ a.e. on $\partial\Omega$, $\beta \in L^\infty(\Omega)$ with $\beta(x)\ge0$ a.e. in $\Omega$ and $\alpha :=  \left(\alpha_{rs}\right)_{1\le r, s\le d} \in {L^{\infty}(\Omega)}^{d \times d}$ is a symmetric diffusion matrix satisfying the uniformly elliptic condition
\begin{align}
\alpha(x)\xi \cdot\xi = \sum_{1\le r,s\le d} \alpha_{rs}(x)\xi_r\xi_s \ge \underline{\alpha} |\xi|^2 \mbox{~a.e. in~} \Omega \label{27-5-19ct1}
\end{align}
for all $\xi = \left(\xi_r\right)_{1\le r\le d} \in \mathbb{R}^d$ with some constant $\underline{\alpha} >0$.

We start with some notations. The expressions $(\cdot,\cdot)_\Omega$ and $(\cdot,\cdot)_{\partial\Omega}$ stand respectively for the scalar products of the Lebesgue spaces $L^2(\Omega)$ and $L^2(\partial\Omega)$, while $[\cdot,\cdot]_{\partial\Omega}$ denotes the dual pair $(\cdot,\cdot)_{\big(H^{-1/2}(\partial\Omega), H^{1/2}(\partial\Omega)\big)}$. Let
\begin{align}\label{30-11-16ct1}
a(u,v) := (\alpha \nabla u, \nabla v)_\Omega + (\beta u,v)_\Omega + (\sigma u, v)_{\partial\Omega} \quad \mbox{and} \quad l^f_j(v) := (f,v)_\Omega + [ j, v]_{\partial\Omega},
\end{align}
where $u,v\in H^1(\Omega)$.

We assume that $\beta(x) \ge \beta_0 >0$ a.e. in $\Omega$ or $\sigma(x)\ge \sigma_0 >0$  a.e. on $\partial\Omega$. Then, there exist positive constants $c_1, c_2$ such that
\begin{align}\label{18-10-16ct1}
c_1\|u\|^2_{1,\Omega} \le a(u,u) \le c_2\|u\|^2_{1,\Omega}
\end{align}
for all $u\in H^1(\Omega)$, where $\|u\|_{1,\Omega} := \left( (\nabla u, \nabla u)_\Omega + ( u,u)_\Omega \right)^{1/2}$ denotes the usual norm of the Sobolev space $H^1(\Omega)$ (cf.\ \cite{mik78}). Thus, the expression $a(u,v)$
generates a scalar product on the space
$ H^1(\Omega)$ equivalent to the usual one.  Consequently, for each $f\in L^2(\Omega)$ the Robin boundary value problem \eqref{17-5-16ct1}
defines a unique weak solution $u$ in the sense that $u := u_j(f) := u(f) \in H^1(\Omega)$ such that the variational equation
\begin{align}\label{17-10-16ct2}
a(u,v)  = l^f_j(v) 
\end{align}
is satisfied for all $v\in H^1(\Omega)$. Furthermore, the estimate
\begin{align}\label{17-10-16ct4}
\|u\|_{1,\Omega} \le \frac{1}{c_1}\max\left(1, \|\gamma\|_{ \mathcal{L}\left(H^1(\Omega), H^{1/2}(\partial\Omega)\right) }\right)\left(\|j\|_{H^{-1/2}(\partial\Omega)} + \|f\|_\Omega\right)
\end{align}
holds true, where $\|f\|_\Omega := \|f\|_{L^2(\Omega)}$ and $\gamma : H^1(\Omega) \to H^{1/2}(\partial\Omega)$ denotes the continuous Dirichlet trace operator. Recall that $\gamma : H^1(\Omega) \to L^2(\partial\Omega)$ is compact (cf.,\ e.g., \cite{lady,Mclean}) and there is a positive constant $c_\gamma$ such that
\begin{align}\label{17-10-16ct4*}
\|\gamma u\|_{\partial\Omega} \le c_\gamma\|u\|_{1,\Omega}
\end{align}
for all $u\in H^1(\Omega)$.
 
In case $\beta=0$ and $\sigma =0$, by the Poincar\'e-Friedrichs inequality (cf.,\ e.g., \cite{Pechstein}) 
$$C_\Omega(v,v)_\Omega \le (\nabla v,\nabla v)_\Omega \quad\mbox{for all}\quad v \in H^1_\diamond(\Omega) := \left\{ u \in H^1(\Omega) ~|~ (u,1)_{\Omega} =0\right\}$$ 
for some constant $C_\Omega >0$ depending only on $\Omega$, the expression $(\alpha \nabla u, \nabla v)_\Omega$ is a scalar product on the space
$ H^1_\diamond(\Omega)$ which is equivalent to the usual one. Therefore, in view of Riesz's representation theorem, for each $f\in L^2(\Omega)$ the problem \eqref{17-5-16ct1} also has a unique weak solution $u :=u_j(f) := u(f) \in H^1_\diamond(\Omega)$ which is defined via the variational equation $(\alpha \nabla u, \nabla v)_\Omega = l^f_j(v)$ for all $v \in H^1_\diamond
(\Omega)$ and further satisfied the estimate $\|u\|_{1,\Omega} \le C\left(\|j\|_{H^{-1/2}(\partial\Omega)} + \|f\|_\Omega\right)$. In the sequel we thus consider the case $\inf_\Omega \beta(x) + \inf_{\partial\Omega}\sigma(x)>0$ only. All results in present paper are still valid for the case $\beta =\sigma
=0$.

The inverse problem is stated as follows:
\begin{center}
{\it Given a boundary measurement $z\in L^2(\Gamma)$ of the exact $g^\dag :=u_{|\Gamma}$ of \eqref{17-5-16ct1}, find $f\in L^2(\Omega)$.}
\end{center}
For this purpose and with the particular interest in reconstructing probably discontinuous sources, we use the standard least squares method with the total variation regularization, i.e. we consider a minimizer of the minimization problem
$$
\min_{f\in F_{ad}} J(f), \quad J(f) := \frac{1}{2}\|u(f)-z\|^2_{\Gamma} + \rho TV(f) \eqno \left(\mathcal{P}\right)
$$
as reconstruction. Here $\rho\in (0,1)$ is the regularization parameter and
$$F_{ad} := \left\{f\in BV(\Omega) ~|~ -\infty<\underline{f} \le f(x) \le \overline{f} <\infty \quad\mbox{for a.e. in} \quad \Omega\right\}$$
with $\underline{f}$, $\overline{f}$ being given constants. We note that the lower and upper bound imposing upon the admissible set $F_{ad}$ guarantees the existence of a minimizer to $\left(\mathcal{P}\right)$ (cf.\ \cite{acar-vogel,Ca99,chavent-kunisch,vogel}). 

Let $u^h$ be the approximation of $u$
in  the  finite dimensional space $\mathcal{V}^h_1 := \left\{v^h\in C(\overline\Omega)
~|~{v^h}_{|T} \in \mathcal{P}_1(T), ~~\forall
T\in \mathcal{T}^h\right\}$ of piecewise linear, continuous finite elements. We then consider the {\it discrete} total variation regularized problem corresponding to $\left(\mathcal{P}\right)$ which is given by the following minimization problem
$$
\min_{f\in F^h_{ad}} J^h(f), \quad J^h(f) := \frac{1}{2}\|u^h(f)-z\|^2_{\Gamma} + \rho TV(f), \eqno \big(\mathcal{P}^h\big),
$$
where $F^h_{ad} := F_{ad} \cap \mathcal{V}_1^h$. Let us point to the reader here that a discrete total variation approach with finite elements has been recently proposed for imaging in \cite{her19}.
Likewise $\left(\mathcal{P}\right)$, we can show that the problem $\big(\mathcal{P}^h\big)$ admits a minimizer $f^h$ for each $h>0$. However, due to the lack of strict convexity of the cost functional, a solution of $\big(\mathcal{P}^h\big)$ may be non-unique. 

As $\rho$ is fixed, we show that every solution sequence $(f^h)_h$ of the problems $\big(\mathcal{P}^h\big)_h$ has a subsequence which converges to a solution of the problem $\left(\mathcal{P}\right)$, in $L^s(\Omega)$-norm as well as in the total variation, where $\forall s\in [1,\infty)$ (see \S3, Theorem \ref{odinh1}). Furthermore, in case $h \to 0$, $\|g^\dag - z\|_{\Gamma}\le \delta_z \to 0$ and $\rho = \rho(h,\delta_z)$ is chosen in a suitable way, every sequence of solutions to $\big(\mathcal{P}^h\big)_h$ can be extracted a subsequence that converges to a sought source with the total variation-minimizing property (see \S 3, Theorem \ref{stability2}). To the best of the author's knowledge those results are new for problem $\big(\mathcal{P}\big)$.

For the numerical solution of problem $\big(\mathcal{P}^h\big)$, we adopt the  linearized primal-dual algorithm proposed in \cite{tian-yuan-16}. For the convenience of the reader we in the following sketch this approach. Starting with the equation
$$\mathbb{T} : BV(\Omega) \to L^2(\Omega)$$
with $\mathbb{T}$ being a bounded linear operator, the authors have considered the unconstrained minimization problem
\begin{align}\label{24-5-19ct1}
\inf_{e \in BV(\Omega)} E(e), \quad E(e) := \frac{1}{2}\|\mathbb{T}e-y\|^2_{\Omega} + \rho TV(e),
\end{align}
where $y$ is the given data. Due to the $TV$-dual representation approach developed by \cite{bartels-12,bartels-15,bartels-16,ChPo11}, \eqref{24-5-19ct1} is rewritten as a saddle-point problem
\begin{align}\label{24-5-19ct2}
\inf_{e \in BV(\Omega)} E(e) = \inf_{e \in BV(\Omega)} \sup_{p \in L^1(\Omega;\mathbb{R}^d)} \mathcal{E}(e,p),
\end{align}
where
\begin{align*}
\mathcal{E}(e,p) := \frac{1}{2}\|\mathbb{T}e-y\|^2_{\Omega} + \rho (\nabla e,p)_\Omega - I_{\mathcal{B}_1(L^1(\Omega;\mathbb{R}^d))}(p),
\end{align*}
$\mathcal{B}_1(L^1(\Omega;\mathbb{R}^d)) := \left\{p = (p_j)_{j=1}^d \in L^1(\Omega;\mathbb{R}^d) ~|~ \|p\|_\infty := \max_{j=1,...,d} \|p_j\|_{L^\infty(\Omega)} \le 1\right\}$,
and $I_{\mathcal{B}_1(L^1(\Omega;\mathbb{R}^d))}$ denotes the indicator function of the set $\mathcal{B}_1(L^1(\Omega;\mathbb{R}^d))$. Using the piecewise constant finite element space $\mathcal{V}_0^h := \left\{v^h\in L^1(\Omega)
~|~{v^h}_{|T} =\mbox{~constant},~\forall
T\in \mathcal{T}^h\right\}$, problem \eqref{24-5-19ct2} is then approximated in the $\left(\mathcal{V}^h_1,\mathcal{V}^h_0\right)$-space according to
\begin{align}\label{24-5-19ct3}
\inf_{e^h \in \mathcal{V}^h_1} \sup_{p^h \in (\mathcal{V}^h_0)^d} 
\left( \frac{1}{2}\|\mathbb{T}e^h - y\|^2_{\Omega} + \rho (\nabla e^h,p^h)_\Omega - I_{\mathcal{B}_1(L^1(\Omega;\mathbb{R}^d))}(p^h)\right).
\end{align}
For the numerical solution of this problem the scheme 
\begin{align}
	  e^h_{n+1} &= \arg \min_{e^h \in \mathcal{V}^h_1} \left\{\left(e^h,  \mathbb{T}^*(\mathbb{T}e^h_n - y)\right)_\Omega +\rho\left( \nabla e^h, p^h_n\right)_\Omega +\frac{1}{2\tau} \|e^h - e^h_n\|^2_\Omega\right\} \label{13-9-18ct1*}\\
	  &\widetilde{e}^h_{n+1} = 2e^h_{n+1}-e^h_n \label{13-9-19ct2*}\\
	  p^h_{n+1} &= \arg \max_{p^h \in (\mathcal{V}^h_0)^d} \left\{ \rho \left( \nabla \widetilde{e}^h_{n+1}, p^h \right)_\Omega - I_{\mathcal{B}_1\left( (\mathcal{V}^h_0)^d\right)}(p^h) -\frac{\theta}{2\tau}\|p^h - p^h_n\|^2_\Omega\right\}   \label{13-9-18ct3*}
      \end{align}
is proposed in \cite[Algorithm 1]{tian-yuan-16}, where the parameters $\tau, \theta$ are chosen suitably. Furthermore, we note that the solution $p^h_{n+1}$ to the sub-problem \eqref{13-9-18ct3*} is given by the explicit form (see \cite{bartels-12,tian-yuan-16})
\begin{align*}
p^h_{n+1} = \frac{p^h_n + \frac{\tau\rho}{\theta} \nabla \widetilde{e}^h_{n+1}}{\max\{1, |p^h_n + \frac{\tau\rho}{\theta} \nabla \widetilde{e}^h_{n+1}|\}}.
\end{align*}
In our framework we consider a constrained minimization problem, i.e. the linear space $\mathcal{V}^h_1$ in \eqref{13-9-18ct1*} is replaced by the bound-constrained set $F^h_{ad}$, with the forward operator being affine linear arising from solutions of elliptic PDEs, and data taken on a part of the boundary. With $\mu^h_0 := (f^h_0, p^h_0) \in F^h_{ad} \times (\mathcal{V}^h_0)^d$ denoting the initial guess the resulting algorithm in our notation is given by
\begin{align}
	  f^h_{n+1} &= \arg \min_{f^h \in F^h_{ad}} \left\{\left( f^h,    u^h_a(f^h_n) \right)_\Omega +\rho\left( \nabla f^h, p^h_n\right)_\Omega +\frac{1}{2\tau} \|f^h - f^h_n\|^2_\Omega\right\} \label{13-9-18ct1}\\
	  &\widetilde{f}^h_{n+1} = 2f^h_{n+1}-f^h_n \label{13-9-19ct2}\\
	  p^h_{n+1} &= \arg \max_{p^h \in (\mathcal{V}^h_0)^d} \left\{ \rho \left( \nabla \widetilde{f}^h_{n+1}, p^h \right)_\Omega - I_{\mathcal{B}_1\left( (\mathcal{V}^h_0)^d\right)}(p^h) -\frac{\theta}{2\tau}\|p^h - p^h_n\|^2_\Omega\right\} \label{13-9-18ct3}
      \end{align}	 
with $u^h_a$ denoting the discrete adjoint state associated to $u^h$, see \S \ref{3-6-19ct1} for details.
We mention that the sub-problem \eqref{13-9-18ct1} admits a solution (cf.\ Remark \ref{31-5-19ct1})
\begin{align*}
f^h_{n+1} = \mathcal{P}^{L^2}_{F^h_{ad}} \left( f^h_n - \tau (u^h_a(f^h_n) - \rho \div~p^h_n)\right),
\end{align*}
where $\mathcal{P}^{L^2}_{F^h_{ad}} : L^2(\Omega) \to F^h_{ad}$ denotes the $L^2$-projection on the set $F^h_{ad}$ and $-\div$ is the adjoint operator of $\nabla$ defined by 
\begin{align}\label{6-12-16ct3}
(-\div~q^h, g^h) = (q^h, \nabla g^h)_\Omega := \sum_{T\in \mathcal{T}^h} \sum_{j=1}^d |T| {q^h_j}_{|_T} {\frac{\partial g^h}{\partial x_j}}{|_T}
\end{align}
for all $q^h := (q^h_j)_{j=1}^d \in (\mathcal{V}^h_0)^d$ and $g^h \in \mathcal{V}^h_1$.

We then show in \S \ref{primal-dual} that the iteration sequence $\left( \mu^h_n\right)_n := \left( f^h_n, p^h_n\right)_n$ admits a subsequence $\left( \mu^h_{n_k}\right)_k$ which converges in the finite dimensional space $\mathcal{V}^h_1 \times (\mathcal{V}^h_0)^d$ to an element $\mu^h_* := (f^h_*,p^h_*) \in F^h_{ad} \times \partial TV(f^h_*)$, where $\mu^h_*$ satisfies the first-order optimality condition for the considered minimization problem and $f^h_*$ is a minimizer of $(\mathcal{P}^h)$. Moreover, we derive the identity
\begin{align*}
\|\mu^h_{n+1}-\mu^h_n\|^2_{\mathcal{V}^h_1 \times (\mathcal{V}^h_0)^d} =\mathcal{O}\left( \frac{1}{n}\right) \text{ for } n \rightarrow \infty,
\end{align*}
which  is to be expected for primal-dual-type algorithms by the results of e.g. \cite{nes05,tian-yuan-16}. We note that using $L^1-$sparsity regularization instead of the $TV-$ regularization the rate $\mathcal{O}\left( \frac{1}{n^2}\right)$ is shown in \cite{nes13}.

Source identification appears in many practical applications, like e.g. electroencephalography, geophysical prospecting and pollutant detection. Its PDE model-based treatment attracted great attention during the last decades, where we refer to e.g. \cite{acosta,bao,chavent2009,Engl_Hanke_Neubauer,hama,isakov89,Tarantola} and the references given there. Let us briefly refer to publications related to our work. In \cite{farcas,matsumoto1,matsumoto2} dual reciprocity boundary element methods are used to treat inverse source problems numerically. The situation where a priori knowledge of the identified source is available, e.g. as a point source, characteristic functional or a harmonic functional, is investigated in \cite{badia1,badia2,batoul,kunisch,ring,trlep}. Recently, by using a so-called energy functional method combined with Tikhonov regularization the authors  in \cite{HHQ17} studied a numerical method for the source identification problem from single Cauchy data.  Finite element analysis  of  the  problem  of  simultaneously  identifying  the source  term  and  coefficients from {\it distributed} observations  can  be  found  in \cite{quyen}. For the treatment of coefficient identification problems employing $TV$-regularization techniques we refer to \cite{ChTa03,ChZo99,Gu90,haqu11,haqu12jmaa,TaLi07}. 

In the present work we study the finite element approximation and algorithm development for the identification problem of a source term in an elliptic PDE from partial observations of the state on the boundary. To the best of our knowledge numerical analysis for the problem setting we consider is not yet available, although there are many contributions to the numerical and algorithmic treatment of source identification problems. The main results of our paper are contained in Theorem \ref{odinh1}, where we prove convergence of the finite element discrete regularized approximations to a solution of the corresponding continuous regularized problem as the mesh size of triangulations tends to zero, in Theorem \ref{stability2}, where we establish convergence to a sought source functional when the regularization parameter approaches to zero with a suitable coupling of noise level and mesh size, and in Theorem \ref{PDA}, where we prove convergence for the iterates of our discrete algorithm.

To conclude this introduction, we briefly present the space of functionals with bounded total variation, for more details the reader may consult \cite{Am00,attouch,EvGa92,Giusti}. A scalar functional $f\in
L^{1}(\Omega)$ is said to be of bounded total variation if
\begin{align*}
TV(f) := \int_{\Omega}\left|\nabla f\right|
:=\sup\left\{\int_{\Omega} f\divv~ G ~\big|~ G\in C^1_c(\Omega)^d,~
\left| G(x)\right|_{\infty} \le
1,~x\in\Omega\right\}<\infty,
\end{align*}
where $\left|\cdot\right|_{\infty}$ denotes the $\ell_{\infty}$-norm
on $\mathbb{R}^d$ defined by
$\left|x\right|_{\infty}=\max\limits_{1\le i\le d}\left|x_i\right|$
and $C^1_c(\Omega)$ is the space of continuously differentiable functionals with compact support in $\Omega$.
The space of all functionals in $L^{1}(\Omega)$ with bounded total
variation 
$$BV(\Omega)=\left\{ f\in L^{1}(\Omega) ~\big|~
TV(f)<\infty\right\}$$
is a {\it non-reflexive} Banach space equipped with the norm
$$
\| f \|_{BV(\Omega)}
:= \| f \|_{L^1(\Omega)} + TV(f).
$$
Furthermore, if $\Omega$ is an open bounded set with
Lipschitz boundary, then $W^{1,1}(\Omega)\varsubsetneq BV(\Omega)$. 

\section{Preliminaries}\label{3-6-19ct1}

Now we summarize some useful properties of the source-to-solution operator $u=u(f)$. First, we note that the decomposition
\begin{align}\label{30-11-16ct2}
u(f) = \overline{u}(f) +\widehat{u}
\end{align}
holds, where $\overline{u}(f)$ and $\widehat{u}$ are the solutions to the variational equations
\begin{align}\label{30-11-16ct3}
a(\overline{u}(f),v) = l^f_0(v) \quad \mbox{and} \quad a(\widehat{u},v) = l^0_j(v)
\end{align}
for all $v\in H^1(\Omega)$, respectively. Thus, the operator $u : L^2(\Omega)\rightarrow H^1(\Omega)$ is continuously Fr\'echet differentiable on $L^2(\Omega)$. For each $f \in L^2(\Omega)$ the action of the Fr\'echet derivative in the direction $\xi \in L^2(\Omega)$ satisfies the equation
\begin{align}\label{30-11-16ct4}
u'(f)\xi = \overline{u}(\xi).
\end{align}
Next, along with \eqref{17-5-16ct1} we consider the {\it adjoint} problem
\begin{equation}\label{17-10-16ct1*}
\begin{aligned}
-\nabla \cdot \big(\alpha \nabla u_a \big) + \beta u_a &= 0 \quad\mbox{in}\quad \Omega, \\
\alpha \nabla u_a \cdot \vec{n} +\sigma u_a &= u(f) -z \quad\mbox{on}\quad \Gamma,\\ 
\alpha \nabla u_a \cdot \vec{n} +\sigma u_a &= 0 \quad\mbox{on}\quad \partial\Omega\setminus\Gamma
\end{aligned}
\end{equation}
which also admits a unique weak solution $u_a(f) \in H^1(\Omega)$ defined via the variational equation
\begin{align}\label{17-10-16ct2*}
a(u_a(f),v)  = (u(f)-z, v)_{\Gamma}
\end{align}
for all $v\in H^1(\Omega)$. Furthermore, for $f,\, \xi \in L^2(\Omega)$ the Fr\'echet differential ${u_a}'(f)\xi$ is the unique weak solution in $H^1(\Omega)$ of the variational equation
\begin{align}\label{30-11-16ct5}
a \left({u_a}'(f)\xi, v\right) = (\overline{u}(\xi),v)_\Gamma \quad\mbox{for all}\quad v\in H^1(\Omega).
\end{align}
Now, for all $f \in L^2(\Omega)$ letting
$$
L(f) := \frac{1}{2}\|u(f)-z\|^2_{\Gamma},
$$
a computation for all $\xi \in L^2(\Omega)$ shows, by \eqref{17-10-16ct2*} and \eqref{30-11-16ct4},
\begin{align*}
L'(f)\xi =(u(f)-z, u'(f)\xi)_\Gamma = a(u_a(f),u'(f)\xi) = a(u_a(f),\overline{u}(\xi))
\end{align*}
and so, by \eqref{30-11-16ct5},
\begin{align}\label{1-12-16ct4}
L''(f)(\xi,\xi) = a({u_a}'(f)\xi,\overline{u}(\xi)) = (\overline{u}(\xi),\overline{u}(\xi))_\Gamma \ge 0.
\end{align}
We however note that the bilinear form $L''(f)$ is in general not positive definite. In fact, considering the particular case $\Omega := \{x = (x_1,x_2)\in\mathbb{R}^2 ~|~ |x| = (x^2_1,x^2_2)^{1/2} < \frac{\sqrt{2}}{2}\}$, $\alpha$ is the unit $d\times d$-matrix, $\beta=0$ and $\sigma = 0$, we have that $\xi = 16 - 64|x|^2 \neq 0$, but it gives $\overline{u}(\xi) = 4|x|^2 - 4|x|^4 - 1$ satisfying $\overline{u}(\xi)_{|\Gamma} = 0$. Therefore, the cost functional $J$ of the problem $(\mathcal{P})$ is {\it not strictly convex}. Furthermore, since the $TV$-regularization term is a semi-norm on the space $BV(\Omega)$ only, it is also {\it not differentiable}.

\begin{lemma}[\cite{Giusti}]\label{bv1}

(i) Let $\left(f_n\right)_n$ be a bounded sequence in the $BV(\Omega)$-norm.
Then a subsequence which is denoted by the same symbol and an element
$f\in BV(\Omega)$ exist such that $\left(f_n\right)_n$ converges to $f$ in the $L^1(\Omega)$-norm.

(ii) Let $\left(f_n\right)_n$ be a sequence in $BV(\Omega)$ converging to
$f$ in  the $L^1(\Omega)$-norm. Then $f \in BV(\Omega)$ and
\begin{align}\label{27-7-16ct1}
TV(f) \le \liminfn TV(f_n).
\end{align}
\end{lemma}

\begin{lemma}\label{weakly conv.}
Assume that the sequence $\left( f_n\right)_n\subset F_{ad}$
converges to an element $f$ in the $L^1(\Omega)$-norm. Then $f\in F_{ad}$, the sequence $\left(u(f_n)\right)_n$ converges to $u(f)$ in the $H^1(\Omega)$-norm and $\left(u(f_n)_{|\Gamma}\right)_n$ converges to $u(f)_{|\Gamma}$ in the $L^2(\Gamma)$-norm as well.
\end{lemma}

\begin{proof}
First, by Lemma \ref{bv1}, we have $f\in BV(\Omega)$. Furthermore, a subsequence $\left( f_{n_m}\right)_m$ exists such that 
\begin{align}\label{29-11-16ct1}
\limm|f_{n_m}(x)-f(x)| =0 \quad \mbox{for a.e. in} \quad \Omega.
\end{align}
Since 
\begin{align}\label{29-11-16ct1*}
\underline{f} \le f_{n_m}(x) \le \overline{f} \quad \mbox{for all} \quad m\in \mathbb{N} \quad \mbox{and a.e. in} \quad \Omega,
\end{align}
sending $m$ to $\infty$, we get $\underline{f} \le f(x) \le \overline{f}$ for a.e. in $\Omega$ which implies that $f\in F_{ad}$.
Now, for all $v\in H^1(\Omega)$ and $m\in \mathbb{N}$ we have from \eqref{17-10-16ct2} that
$$a\left(u(f_{n_m}) - u(f),v\right)  = (f_{n_m}-f,v)_\Omega$$
which together with \eqref{18-10-16ct1} yield
$$c_1\left\|u(f_{n_m}) - u(f)\right\|_{1,\Omega} \le \|f_{n_m} - f\|_\Omega \le \sqrt{2\overline{f}} \|f_{n_m} - f\|^{1/2}_{L^1(\Omega)}$$
that deduces
$\limm\left\|u(f_{n_m}) - u(f)\right\|_{1,\Omega}=0$.
Since the limit is unique, the whole sequence $\left(u(f_n)\right)_n$ converges to $u(f)$ in the $H^1(\Omega)$-norm also. This completes the proof.
\end{proof}

\begin{lemma}\label{existance}
The minimization problem $\left(\mathcal{P}\right)$ admits a minimizer.
\end{lemma}

\begin{proof}
The assertion follows directly from Lemma \ref{bv1} and Lemma \ref{weakly conv.}, therefore omitted here (see \cite{acar-vogel} for details).
\end{proof}

\section{Discretization and convergence}\label{discrete}

Let $\left(\mathcal{T}^h\right)_{0<h<1}$ be a
quasi-uniform family of regular triangulations of the domain $\overline{\Omega}$ with the mesh size $h$ such that each vertex of the polygonal boundary $\partial\Omega$ is a node of $\mathcal{T}^h$ (cf.\ \cite{Brenner_Scott,Ciarlet}).
For the definition of the discretization space of the state
functionals let us denote the piecewise affine, globally continuous finite element space by
\begin{equation*}
\mathcal{V}_1^h := \left\{v^h\in C(\overline\Omega)
~|~{v^h}_{|T} \in \mathcal{P}_1(T), ~~\forall
T\in \mathcal{T}^h\right\},
\end{equation*}
where $\mathcal{P}_1$ consists all polynomial functionals of degree
less than or equal to 1. The piecewise constant finite element space is  denoted by
\begin{equation*}
\mathcal{V}_0^h := \left\{v^h\in L^1(\Omega)
~|~{v^h}_{|T} = \mbox{~constant}, ~~\forall
T\in \mathcal{T}^h\right\}.
\end{equation*}
The admissible set is discretized as
\begin{equation*}
F^h_{ad} := F_{ad} \cap \mathcal{V}_1^h.
\end{equation*}

Similar to the continuous case, we for each $f\in L^2(\Omega)$ have that the equation
\begin{align}
a(u^h,v^h)  = l^f_j(v^h)  \quad \mbox{for all} \quad v^h\in
\mathcal{V}_{1}^h \label{10/4:ct1}
\end{align}
admits a unique solution $u^h :=u^h_j(f) := u^h(f) \in \mathcal{V}_{1}^h$ which further satisfies the estimate
\begin{align}
\|u^h \|_{H^1(\Omega)}\le C \left( \|f\|_{\Omega}+
\|j\|_{H^{-1/2}(\partial\Omega)} \right),  \label{18/5:ct1}
\end{align}
where the constant $C$ is independent of the mesh size $h$. Then the problem $\left(\mathcal{P}\right)$ can be discretized by
$$
\min_{f\in F^h_{ad}} J^h(f), \quad J^h(f) := \frac{1}{2}\|u^h(f)-z\|^2_{\Gamma} + \rho TV(f). \eqno \big(\mathcal{P}^h\big)
$$
which admits a minimizer $f^h$ for each $h>0$. 

Also, the adjoint state $u_a=u_a(f)$ in \eqref{17-10-16ct1*} has the discrete version $u^h_a =u^h_a(f)$ defined via the equation
\begin{align}\label{17-10-16ct2**}
a(u^h_a(f),v^h)  = (u^h(f)-z, v^h)_{\Gamma}
\end{align}
for all $v^h\in \mathcal{V}^h_{1}$.

We now state the following result on the stability of finite element approximations. Here and in the sequel, unless otherwise stated, we indicate by $C$ a generic positive constant which is independent of the mesh size $h$, the regularization parameter $\rho$ and the observation data $z$. 

\begin{theorem}\label{odinh1}
Let $(h_n)_n$ be a sequence with $\limn h_n = 0$ and $\rho$ be a fixed regularization parameter. Assume that
$f^{h_n} \in F^{h_n}_{ad}$ is an arbitrary minimizer of
$\left(\mathcal{P}^{h_n}\right)$ for each $n\in\mathbb{N}$.
Then a subsequence of $\big(f^{h_n}\big)_n$ not relabeled and an element $f \in F_{ad}$ exist such that
\begin{align}\label{eq:conv}
\limn \big\| f^{h_n} - f\big\|_{L^1(\Omega)} = 0 \quad \mbox{and} \quad \limn TV( f^{h_n}) = TV(f).
\end{align}
Furthermore, $f$ is a solution to $\left( \mathcal{P} \right)$. If the uniqueness of the solution to $\left( \mathcal{P} \right)$ is satisfied, then convergence \eqref{eq:conv} holds for the whole sequence.
\end{theorem}

We mention that \eqref{eq:conv} yields the convergence in the $L^s(\Omega)$-norm for all $1\le s<\infty$. In fact, this assertion follows directly from the following estimate
\begin{align*}
\|f^{h_n} -f\|^s_{L^s(\Omega)} = \int_\Omega |f^{h_n} -f|\cdot|f^{h_n} -f|^{s-1} \le \int_\Omega |f^{h_n} -f|\cdot\left(|f^{h_n}|+|f|\right)^{s-1} \le \left(2\overline{f}\right)^{s-1}\|f^{h_n} -f\|_{L^1(\Omega)}.
\end{align*}
To prove Theorem \ref{odinh1} we need the following crucial result.

\begin{lemma} \cite[Lemma 4.6.]{HKT}\label{HKTlemma}
For any fixed $\widehat{f}\in F_{ad}$ an element $\widehat{f}^h \in F^h_{ad}$ exists such that
\begin{align}\label{22-7-16ct5}
\big\|\widehat{f}^h-\widehat{f}\big\|_{L^1(\Omega)} \le Ch|\log h|
\end{align}
and
\begin{align}\label{22-7-16ct6}
\lim_{h\to 0} TV(\widehat{f}^h) = TV(\widehat{f}).
\end{align} 
\end{lemma}

\begin{lemma} \label{30-11-16ct6}
Assume that $(f_n)_n \subset F_{ad}$ converges to $f\in F_{ad}$ in the $L^1(\Omega)$-norm. Then the limit
\begin{align} \label{20-5-16ct3}
\limn \|u^{h_n}(f_n) - z\|_\Gamma = \|u(f)-z\|_\Gamma
\end{align}
holds true.
\end{lemma}

\begin{proof}
The proof is based on standard arguments, it is therefore omitted here. 
\end{proof}

\begin{proof}[Proof of Theorem \ref{odinh1}]
Let $\widehat{f}\in F_{ad}$ be arbitrary and $\widehat{f}^{h_n}$ be generated from $\widehat{f}$ due to Lemma \ref{HKTlemma}. The optimality of $f^{h_n}$, yields
$$\frac{1}{2} \|u^{h_n}(f^{h_n}) - z\|^2_\Gamma +\rho TV(f^{h_n}) \le \frac{1}{2}\|u^{h_n}(\widehat{f}^{h_n}) - z\|^2_\Gamma +\rho TV(\widehat{f}^{h_n}).$$
By \eqref{18/5:ct1} and \eqref{22-7-16ct6}, we deduce from the last inequality that the sequence $\big(f^{h_n}\big)_n$ is bounded in the $BV(\Omega)$-norm. Then, due to Lemma \ref{bv1}, a subsequence not relabeled and an element $f \in F_{ad}$ exist such that
\begin{align}\label{30-11-16ct8}
\limn \big\| f^{h_n} - f\big\|_{L^1(\Omega)} = 0 \quad \mbox{and} \quad TV(f) \le \liminfn TV( f^{h_n}).
\end{align}
Combining this with Lemma \ref{30-11-16ct6}, we arrive at
\begin{align} \label{30-11-16ct7}
\frac{1}{2}\|u(f)-z\|^2_\Gamma +\rho TV(f) 
&\le  \limn \frac{1}{2}\|u^{h_n}(f^{h_n}) - z\|^2_\Gamma + \liminfn \rho TV( f^{h_n}) \notag\\
&= \liminfn \left( \frac{1}{2}\|u^{h_n}(f^{h_n}) - z\|^2_\Gamma + \rho TV( f^{h_n})\right) \notag\\
&\le \limsupn \left( \frac{1}{2}\|u^{h_n}(f^{h_n}) - z\|^2_\Gamma + \rho TV( f^{h_n})\right)\notag\\
&\le \limsupn \left( \frac{1}{2}\|u^{h_n}(\widehat{f}^{h_n}) - z\|^2_\Gamma +\rho TV(\widehat{f}^{h_n})\right) \notag\\
&= \frac{1}{2}\|u(\widehat{f})-z\|^2_\Gamma +\rho TV(\widehat{f}),
\end{align}
here we used Lemma \ref{HKTlemma} (with noting $\limn h_n|\log h_n| =0$) and Lemma \ref{30-11-16ct6} in the last equation, again. Thus, $f$ is a solution to $\left( \mathcal{P} \right)$. Furthermore, replacing $\widehat{f}$ in \eqref{30-11-16ct7} by $f$, we obtain 
\begin{align*}
\frac{1}{2}\|u(f)-z\|^2_\Gamma +\rho TV(f) 
&= \limsupn \left( \frac{1}{2}\|u^{h_n}(f^{h_n}) - z\|^2_\Gamma + \rho TV( f^{h_n})\right)\\
&= \limn \frac{1}{2}\|u^{h_n}(f^{h_n})- z\|^2_\Gamma + \limsupn \rho TV( f^{h_n})\\
&= \frac{1}{2}\|u(f)-z\|^2_\Gamma+ \limsupn \rho TV( f^{h_n})
\end{align*}
and, together with \eqref{30-11-16ct8}, arrive at
$$TV(f) \le \liminfn TV( f^{h_n}) \le \limsupn TV( f^{h_n}) = TV(f),$$
and so that $TV(f) = \limn TV( f^{h_n})$. The proof is completed.
\end{proof}

In the remaining part of this section we investigate the convergence of discrete regularized approximations to an identified source as $\rho$ is chosen in a suitable way depending on the mesh size and the error level of observations. Before doing so, we introduce the notion of the total variation-minimizing solution of the identification problem.

\begin{lemma}\label{identification_problem}
The problem
$$
\min_{f \in \mathcal{I}\left(g^\dag\right) :=\left\{ f \in F_{ad} ~\big|~ u(f)_{|\Gamma} =  g^\dag \right\} }  TV(f) \eqno\left(\mathcal{IP}\right)
$$
admits a solution, called the {\it total variation-minimizing solution} of the identification problem.
\end{lemma}

\begin{proof}
The assertion follows by stand arguments, therefore omitted here.
\end{proof}

Let $f\in L^2(\Omega)$ be fixed, we denote by
\begin{align}\label{3-6-16ct9}
\chi^h_f := \big\|u^h(f) -u(f)\big\|_{1,\Omega}.
\end{align}
Then (cf.\ \cite{Brenner_Scott,Ciarlet}),
\begin{align}\label{21-5-16ct1}
\lim_{h\to 0} \chi^h_f=0 \mbox{\quad and \quad} 0\le \chi^h_f\le Ch \mbox{\quad if \quad} u(f) \in H^2(\Omega).
\end{align}

We now show the convergence of finite element approximations to the identification. 

\begin{theorem}\label{stability2}
Assume that
$\limn h_n = 0$ and $(\delta_n)_n$ and $(\rho_n)_n$ be any positive sequences such that
\begin{align}\label{31-8-16ct1}
\rho_n \to 0, \quad \frac{\delta_n}{\sqrt{\rho_n}} \to 0, \quad \frac{ h_n|\log h_n|}{\rho_n} \to 0 \quad \mbox{and} \quad \frac{\chi^{h_n}_{\widehat{f}}}{\sqrt{\rho_n}} \to 0 \quad \mbox{as} \quad n\to\infty,
\end{align}
where $\widehat{f} \in \mathcal{I} \left(g^\dag\right)$ is arbitrary and $\chi^{h_n}_{\widehat{f}}$ is defined by \eqref{3-6-16ct9}. Furthermore, assume that
$\left( z_n\right)_n \subset L^2(\Gamma)$ is a sequence satisfying
$$\big\|z_n - g^\dag\big\|_{\Gamma} \le \delta_n$$
and $f_n $ is an arbitrary minimizer of the problem
$$\min_{f\in F^{h_n}_{ad}} \frac{1}{2}\|u^{h_n}(f)-z_n\|^2_{\Gamma} + \rho_n TV(f)$$
for each $n\in N$. Then, a subsequence of $(f_n)_n$ not relabeled and a total variation-minimizing solution $f^\dag$ of the identification problem $\left(\mathcal{IP}\right)$ exist such that
\begin{align}\label{9-12-16ct1}
\limn \|f_n -f^\dag\|_{L^1(\Omega)} =0 \quad \mbox{and} \quad \limn TV(f_n) =TV(f^\dag).
\end{align}
Furthermore, the discrete state sequence $\big( u^{h_n}(f_n)\big)_n$ converges in the $H^1(\Omega)$-norm to the unique weak solution $u(f^\dag)$ of the boundary value problem \eqref{17-5-16ct1}.
\end{theorem}

\begin{proof}
By the definition of $f_n$, we have
\begin{align}\label{6-5-16ct8}
\frac{1}{2}\big\|u^{h_n}(f_n)-z_n \big\|^2_{\Gamma} + \rho_n TV(f_n) &\le \frac{1}{2}\big\|u^{h_n}(\widehat{f}^{h_n})-z_n \big\|^2_{\Gamma} + \rho_n TV(\widehat{f}^{h_n}),
\end{align}
where $\widehat{f}^{h_n} \in F^{h_n}_{ad}$ generates from $\widehat{f}$, due to Lemma \ref{HKTlemma}. We bound
\begin{align}\label{24-10-16ct1}
\frac{1}{2}\big\|u^{h_n}(\widehat{f}^{h_n})-z_n \big\|^2_{\Gamma} 
&\le \big\|u^{h_n}(\widehat{f}^{h_n})- u(\widehat{f}) \big\|^2_{\Gamma} + \big\|u(\widehat{f})- z_n\big\|^2_{\Gamma}  \nonumber\\
&\le C \big\|u^{h_n}(\widehat{f}^{h_n})- u(\widehat{f}) \big\|^2_{1,\Omega} + \delta^2_n.
\end{align}
Using \eqref{17-10-16ct2} and \eqref{10/4:ct1}, we get for all $v^{h_n}\in \mathcal{V}^{h_n}_1$ that
\begin{align*}
a\left( u(\widehat{f})-u^{h_n}(\widehat{f}^{h_n}), v^{h_n}\right) = \left( \widehat{f} - \widehat{f}^{h_n}, v^{h_n}\right)_\Omega
\end{align*}
and so that
\begin{align*}
a\left( u^{h_n}(\widehat{f})-u^{h_n}(\widehat{f}^{h_n}), v^{h_n}\right) = \left( \widehat{f} - \widehat{f}^{h_n}, v^{h_n}\right)_\Omega + a\left( u^{h_n}(\widehat{f})-u(\widehat{f}), v^{h_n}\right).
\end{align*}
Taking $v^{h_n} =u^{h_n}(\widehat{f})-u^{h_n}(\widehat{f}^{h_n})$, by \eqref{18-10-16ct1}, \eqref{22-7-16ct5} and \eqref{3-6-16ct9}, we deduce
\begin{align*}
c_1\big\|u^{h_n}(\widehat{f})- u^{h_n}(\widehat{f}^{h_n}) \big\|_{1,\Omega} 
&\le \big\|\widehat{f} - \widehat{f}^{h_n}\big\|_\Omega + C\big\|u^{h_n}(\widehat{f})-u(\widehat{f})\big\|_{1,\Omega}\\
&\le C\left( h_n|\log h_n|\right)^{1/2} + C\chi^{h_n}_{\widehat{f}}
\end{align*}
and then
\begin{align*}
\big\|u^{h_n}(\widehat{f}^{h_n})- u(\widehat{f}) \big\|_{1,\Omega} \le \big\|u(\widehat{f}) - u^{h_n}(\widehat{f})\big\|_{1,\Omega} + \big\|u^{h_n}(\widehat{f})- u^{h_n}(\widehat{f}^{h_n}) \big\|_{1,\Omega} \le C\left( \left( h_n|\log h_n|\right)^{1/2} + \chi^{h_n}_{\widehat{f}}\right).
\end{align*}
Combining this with \eqref{24-10-16ct1}--\eqref{6-5-16ct8}, we get
\begin{align}\label{6-5-16ct8*}
\frac{1}{2}\big\|u^{h_n}(f_n)-z_n \big\|^2_{\Gamma} + \rho_n TV(f_n) \le C\left( \delta_n^2 + h_n|\log h_n| + \left( \chi^{h_n}_{\widehat{f}}\right)^2\right) + \rho_n TV(\widehat{f}^{h_n}).
\end{align}
It follows from the last inequality and \eqref{31-8-16ct1}, \eqref{22-7-16ct6} that
\begin{align}\label{6-5-16ct10}
\limn \big\|u^{h_n}(f_n)-z_n \big\|_{\Gamma}  =0
\end{align}
and
\begin{align}\label{6-5-16ct11}
\limsupn TV(f_n) \le  TV(\widehat{f}).
\end{align}
Due to Lemma \ref{bv1} and Lemma \ref{30-11-16ct6}, a subsequence of $\left( f_n\right)_n $ not relabeled and an element $f^\dag \in F_{ad}$ exist such that 
\begin{align}
& f_n \rightarrow f^\dag \quad \mbox{~ in the ~} L^1(\Omega)\mbox{-norm}, \nonumber\\
& TV(f^\dag) \le \liminfn TV(f_n), \label{6-5-16ct12}\\
& u^{h_n}(f_n) \rightarrow u(f^\dag) \quad\mbox{~ in the ~} L^2(\Gamma)\mbox{-norm}.\notag
\end{align}
Thus, it follows from \eqref{6-5-16ct10} that
\begin{align*}
\|u(f^\dag)-g^\dag\|_{\Gamma} \le \limn \left( \|u(f^\dag)-u^{h_n}(f_n)\|_\Gamma + \|u^{h_n}(f_n)-z_n\|_\Gamma + \|z_n-g^\dag\|_\Gamma\right)  =0
\end{align*}
and so that $f^\dag \in \mathcal{I} \left(g^\dag\right)$. Furthermore, by \eqref{6-5-16ct11} and \eqref{6-5-16ct12}, we also get
$$TV(f^\dag) \le \liminfn TV(f_n) \le \limsupn TV(f_n) \le  TV(\widehat{f})$$
for any $\widehat{f} \in \mathcal{I} \left(g^\dag\right)$. Therefore, $f^\dag$ is a total variation-minimizing solution of the identification problem. Now, by setting $\widehat{f} =f^\dag$,
it implies that
$$TV(f^\dag) = \limn  TV(f_n).$$
Finally, in view of \eqref{10/4:ct1} and \eqref{18-10-16ct1} we arrive at
\begin{align*}
c_1 \|u^{h_n}(f_n) - u^{h_n}(f^\dag)\|_{1,\Omega} \le \|f_n-f^\dag\|_\Omega \le \left( 2\overline{f} \|f_n-f^\dag\|_{L^1(\Omega)}\right)^{1/2} \to 0 \quad\mbox{as}\quad n\to \infty
\end{align*}
and therefore conclude that 
$$\limn \big\|u^{h_n}(f_n) - u(f^\dag)\big\|_{1,\Omega}
\le \limn \|u^{h_n}(f_n) - u^{h_n}(f^\dag)\|_{1,\Omega} + \limn \|u^{h_n}(f^\dag) - u(f^\dag)\|_{1,\Omega} = 0,$$
which finishes the proof.
\end{proof}

\section{An iterated total variation algorithm}\label{primal-dual}

The aim of this section is to propose an algorithm to reach minimizers of the problem $\big(\mathcal{P}^h\big)$. We start with the following note.  

\begin{remark} (i) Any $\phi \in \mathcal{V}^h_1$ can be considered as an
element of ${(\mathcal{V}^h_1)}^*$, the dual space of $\mathcal{V}^h_1$, by
\begin{align}\label{1-12-16ct7}
(\phi,\xi^h)_{\left( {(\mathcal{V}^h_1)}^*, \mathcal{V}^h_1\right) } := (\phi,\xi^h)_\Omega \quad \mbox{for all} \quad \xi^h \in \mathcal{V}^1_h.
\end{align}
(ii) The inclusion $(\mathcal{V}^h_0)^d \subset {(\mathcal{V}^h_1)}^*$ holds via the identity
\begin{align}\label{1-12-16ct7*}
(p^*,\xi^h)_{\left( {(\mathcal{V}^h_1)}^*, \mathcal{V}^h_1\right) } := (\nabla \xi^h,p^*)_\Omega \quad \mbox{for all} \quad p^*\in (\mathcal{V}^h_0)^d \quad \mbox{and} \quad \xi^h \in \mathcal{V}^1_h.
\end{align}
\end{remark}

\begin{lemma}\label{subgradient}
For each $f^h\in\mathcal{V}^h_1$ the relation
\begin{align}\label{1-12-16ct1}
\partial TV(f^h) = \left\{ p^h\in \mathcal{B}_1\left( (\mathcal{V}^h_0)^d\right) \subset (\mathcal{V}^h_0)^d ~\big|~  (\nabla f^h, p^h)_\Omega = \int_\Omega |\nabla f^h| \right\}
\end{align}
holds, where
$$\mathcal{B}_1\left( (\mathcal{V}^h_0)^d\right) := \left\{ p^h \in (\mathcal{V}^h_0)^d ~|~  \|p^h\|_\infty \le 1\right\}$$
and for $p^h := (p^h_j)_{j=1}^d\in (\mathcal{V}^h_0)^d$, $ \|p^h\|_\infty := \max_{j=1,...,d}\|p^h_j\|_{L^\infty(\Omega)}$.
\end{lemma}

\begin{proof}
Let $f^h\in\mathcal{V}^h_1$ be fixed. As shown in \cite{bartels-12} that
\begin{align}\label{25-9-18ct1}
TV(f^h) = \int_\Omega |\nabla f^h| = \max_{p^h\in (\mathcal{V}^h_0)^d,~\|p^h\|_\infty \le 1} (\nabla f^h, p^h)_\Omega.
\end{align}
Assume that $\mathcal{T}^h := \{T_1, T_2, ..., T_{n^h}\}$ with $n^h\in\mathbb{N}$ and $\overline{\Omega} = \cup_{i=1}^{n^h} \overline{T_i}$. We then consider $p^* := (p^*_{ij})_{n^h\times d} \in \mathcal{B}_1\left( (\mathcal{V}^h_0)^d\right)$ with, for all $i=1,...,n^h$ and $j=1,...,d$,
\begin{equation*}
p^*_{ij} := 
\begin{cases}
+1 \quad &\mbox{if}\quad \frac{\partial {f^h}_{|T_i}}{x_j}>0,\\
-1 \quad &\mbox{if}\quad \frac{\partial {f^h}_{|T_i}}{x_j}<0,\\
\in [-1, +1] \quad &\mbox{if}\quad \frac{\partial {f^h}_{|T_i}}{x_j}=0.
\end{cases}
\end{equation*}
and deduce from \eqref{1-12-16ct7*}, \eqref{25-9-18ct1} that
$$TV(f^h) = (p^*,f^h)_{\left( {(\mathcal{V}^h_1)}^*, \mathcal{V}^h_1\right) }.$$
The subdifferential $\partial TV(f^h)$ is then given by (cf.\ \cite[Theorem 2.4.18]{zal02})
\begin{align*}
\partial TV(f^h) = \mathbf{Co} \left(  \bigcup \left\{p^* ~\big|~ (p^*,f^h)_{\left( {(\mathcal{V}^h_1)}^*, \mathcal{V}^h_1\right) } = TV(f^h) \quad \mbox{and} \quad \|p^*\|_\infty \le 1 \right\} \right),
\end{align*}
and so that \eqref{1-12-16ct1} follows, which finishes the proof.
\end{proof}

We mention that, similar to the decomposition \eqref{30-11-16ct2}, for all $f\in L^2(\Omega)$ we have
\begin{align}\label{30-11-16ct2*}
u^h(f) = \overline{u^h}(f) +\widehat{u^h},
\end{align}
where
\begin{align}\label{30-11-16ct3*}
a\left( \overline{u^h}(f),v^h\right)  = l^f_0(v^h) \quad \mbox{and} \quad a\left( \widehat{u^h},v^h\right)  = l^0_j(v^h)
\end{align}
for all $v^h\in \mathcal{V}^h_1$. Furthermore, the identity 
\begin{align}\label{30-11-16ct4*}
{u^h}'(f)\xi = \overline{u^h}(\xi)
\end{align}
holds for all $f,\xi \in L^2(\Omega)$.

\begin{lemma}\label{first-order-condition}
The functional $f^h \in F^h_{ad}$ is a solution of $(\mathcal{P}^h)$ if and only if there exists $p^h\in \partial TV(f^h)$ such that
\begin{align}
\left( g^h-f^h, u^h_a(f^h)\right)_\Omega + \rho \left( \nabla (g^h-f^h), p^h\right)_\Omega &\ge 0 \quad \mbox{for all} \quad g^h\in F^h_{ad}, \label{1-12-16ct2*}\\
\left( \nabla f^h, q^h - p^h\right)_\Omega &\le 0 \quad \mbox{for all} \quad q^h\in (\mathcal{V}^h_0)^d, \quad \|q^h\|_\infty \le 1
\label{1-12-16ct2**}
\end{align}
where $u^h_a$ is the discrete adjoint state defined by \eqref{17-10-16ct2**}.
\end{lemma}

\begin{proof}
The functionals $\|u^h(\cdot)-z\|_\Gamma$ and $TV(\cdot)$ are both convex on the convex set $F^h_{ad}$, an element $f^h \in F^h_{ad}$ is thus a solution of $(\mathcal{P}^h)$ if and only if
\begin{align}\label{1-12-16ct5}
\left( u^h(f^h) -z, {u^h}'(f^h)(g^h-f^h)\right)_\Gamma +\rho\left( TV(g^h) - TV(f^h)\right) \ge 0
\end{align}
for all $g^h\in F^h_{ad}$. Using \eqref{17-10-16ct2**} and \eqref{30-11-16ct4*}--\eqref{30-11-16ct3*}, we get
\begin{align*}
\left( u^h(f^h) -z, {u^h}'(f^h)(g^h-f^h)\right)_\Gamma 
&= a\left( u^h_a(f^h), {u^h}'(f^h)(g^h-f^h)\right)\\
&= a\left( u^h_a(f^h), \overline{u^h}(g^h-f^h)\right)\\
&=\left( g^h-f^h, u^h_a(f^h)\right)_\Omega
\end{align*}
and thus deduce from \eqref{1-12-16ct5} that
\begin{align}\label{1-12-16ct6}
\left( g^h, u^h_a(f^h)\right)_\Omega +\rho TV(g^h) \ge \left( f^h, u^h_a(f^h)\right)_\Omega +\rho TV(f^h), \quad \forall g^h\in F^h_{ad}.
\end{align}
We here consider the convex functional
$$\Psi : \mathcal{V}^h_1 \to \mathbb{R}\cup\{\infty\} \quad \mbox{with} \quad \xi^h \mapsto \Psi(\xi^h) := \left( \xi^h, u^h_a(f^h)\right)_\Omega +\rho TV(\xi^h) + I_{F^h_{ad}} (\xi^h),$$
where $I_{F^h_{ad}}$ is the indicator functional of the set $F^h_{ad}$. Then, the inequality \eqref{1-12-16ct6} is equivalent to the relation
\begin{align}\label{1-12-16ct6*}
0 \in \partial \Psi(f^h).
\end{align}
Since, for all $\xi^h \in \mathcal{V}^h_1$
$$\partial I_{F^h_{ad}}(\xi^h) = \left\{ p^*\in {(\mathcal{V}^h_1) }^* ~\big|~ (p^*, g^h-\xi^h)_{\left( {(\mathcal{V}^h_1)}^*, \mathcal{V}^h_1\right) } \le 0 \quad \mbox{for all} \quad g^h\in F^h_{ad}\right\},$$
the relation \eqref{1-12-16ct6*} means that there exists $p^h\in \partial TV(f^h)$ such that
$$\left( - u^h_a(f^h) - \rho p^h, g^h-f^h\right)_{\left( {(\mathcal{V}^h_1)}^*, \mathcal{V}^h_1\right) } \le 0 $$
for all $g^h\in F^h_{ad}$. This together with \eqref{1-12-16ct7}--\eqref{1-12-16ct7*} yields \eqref{1-12-16ct2*}. Now, for all $q^h\in (\mathcal{V}^h_0)^d$ with $\|q^h\|_\infty \le 1$ we have from the fact $p^h\in \partial TV(f^h)$ defined by the equation \eqref{1-12-16ct1} that
\begin{align*}
\left( \nabla f^h, q^h\right)_\Omega = \int_\Omega \nabla f^h\cdot q^h \le \int_\Omega |\nabla f^h|\cdot|q^h| \le \int_\Omega |\nabla f^h| = \left( \nabla f^h, p^h\right)_\Omega,
\end{align*}
which finishes the proof.
\end{proof}

\begin{remark}\label{16-4-18ct3}
The system \eqref{1-12-16ct2*}--\eqref{1-12-16ct2**} is equivalent to the following inequality
\begin{align}\label{1-12-16ct2}
\left( g^h-f^h, u^h_a(f^h)\right)_\Omega + \rho \left( \nabla (g^h-f^h), p^h\right)_\Omega - \left( \nabla f^h, q^h - p^h\right)_\Omega &\ge 0
\end{align}
for all $(g^h, q^h) \in F^h_{ad} \times (\mathcal{V}^h_0)^d$ with $\|q^h\|_\infty \le 1$.
\end{remark}
In the sequel, we make use the following notation (cf.\ \cite{bartels-12, bartels-16})
\begin{align}\label{16-4-18ct1}
\|\nabla\| := \sup_{0\neq v^h \in \mathcal{V}^h_1} \dfrac{\|\nabla v^h\|_\Omega}{\|v^h\|_\Omega}.
\end{align}
By the inverse inequality (cf.\ \cite{Brenner_Scott,Ciarlet}) $\|v^h\|_{1,\Omega} \le Ch^{-1}\|v^h\|_\Omega$ for all $v^h \in \mathcal{V}^h_1$, we obtain that $\|\nabla\| \le Ch^{-1}$.

\begin{algorithm}[H] \label{31-8-18ct1}
    \SetKwInOut{Input}{Input}
    \SetKwInOut{Output}{Output}
     \vspace{0.3cm}
    \Input{Let a desired number of iterations $N$ and a regularization parameter $\rho>0$ be given. Choose an initial iterate $\mu^h_0 := (f^h_0, p^h_0) \in F^h_{ad} \times (\mathcal{V}^h_0)^d$ and parameters $\tau, \theta>0$ such that
    \begin{align}\label{5-12-16ct0}
     \left( \frac{1}{\tau} - \frac{c^2_\gamma}{c^2_1}\right)\frac{\theta}     {\tau} > \rho^2\|\nabla\|^2,
    \end{align}
where $c_1$ and $c_\gamma$ are respectively defined by \eqref{18-10-16ct1} and \eqref{17-10-16ct4*}.    
    Set $n=0$.}
    \Output{An approximation of a solution of $(\mathcal{P}^h)$}
    \vspace{0.3cm}
    \If{$n\le N$}
      {\eqref{13-9-18ct1}, \eqref{13-9-19ct2} and \eqref{13-9-18ct3}        
      }
    \caption{Minimizing of $(\mathcal{P}^h)$}
\end{algorithm}

\begin{remark} \label{31-5-19ct1}
Due to the optimality of $f^h_{n+1}$ in \eqref{13-9-18ct1}, we have that
\begin{align*}
\left( u^h_a(f^h_n), g^h-f^h_{n+1}\right)_\Omega +\rho \left( \nabla(g^h -f^h_{n+1}), p^h_n\right)_\Omega +\frac{1}{\tau}\left( f^h_{n+1} -f^h_n, g^h - f^h_{n+1}\right)_\Omega \ge 0
\end{align*}
for all $g^h \in F^h_{ad}$ which can be rewritten as
\begin{align*}
\left( f^h_n +\tau (\rho \div~p^h_n - u^h_a(f^h_n)) - f^h_{n+1}, g^h- f^h_{n+1}\right)_\Omega \le 0 \quad \mbox{for all} \quad g^h \in F^h_{ad}
\end{align*}
or
\begin{align}\label{3-9-18ct2}
f^h_{n+1} = \mathcal{P}^{L^2}_{F^h_{ad}} \left( f^h_n - \tau (u^h_a(f^h_n) - \rho \div~p^h_n)\right),
\end{align}
where $\mathcal{P}^{L^2}_{F^h_{ad}} : L^2(\Omega) \to F^h_{ad}$ denotes the $L^2$-projection on the set $F^h_{ad}$ characterized by
\begin{align*}
\left( \phi - \mathcal{P}^{L^2}_{F^h_{ad}} \phi, g^h - \mathcal{P}^{L^2}_{F^h_{ad}} \phi\right)_\Omega \le 0  \quad \mbox{for all} \quad g^h \in F^h_{ad}
\end{align*}
for each $\phi\in L^2(\Omega)$.
\end{remark}

\begin{remark}
For all $f\in L^2(\Omega)$ let $\overline{u^h_a}(f)\in \mathcal{V}^h_1$ be the unique weak solution of the following variational equation 
\begin{align} \label{6-12-16ct1}
a\left( \overline{u^h_a}(f),v^h\right)  = \left( \overline{u^h}(f), v^h\right)_\Gamma 
\end{align}
for all $v^h\in \mathcal{V}^h_1$, where $\overline{u^h}(f)$ is defined by \eqref{30-11-16ct3*}. Then we have for all $v^h\in \mathcal{V}^h_1$ and $f,g\in L^2(\Omega)$ that
\begin{align*}
a\left( u^h_a(f) -u^h_a(g), v^h\right) 
&= a\left( u^h_a(f), v^h\right) - a\left(u^h_a(g), v^h\right) \\
&= \left( u^h(f)-u^h(g), v^h\right) _{\Gamma}, \quad \mbox{by} ~ \eqref{17-10-16ct2**}\\
&= \left( \overline{u^h}(f)-\overline{u^h}(g), v^h\right) _{\Gamma}, \quad \mbox{by} ~ \eqref{30-11-16ct2*}\\
&= \left( \overline{u^h}(f-g), v^h\right) _{\Gamma}\\
&=a\left( \overline{u^h_a}(f-g), v^h\right), \quad \mbox{by} ~ \eqref{6-12-16ct1}.
\end{align*}
Thus, by \eqref{18-10-16ct1}, taking $v^h = u^h_a(f) -u^h_a(g) - \overline{u^h_a}(f-g)\in \mathcal{V}^h_1$, we obtain
\begin{align}\label{6-12-16ct2}
 u^h_a(f) -u^h_a(g) = \overline{u^h_a}(f-g)
\end{align}
for all $f,g\in L^2(\Omega)$. Furthermore, by \eqref{18-10-16ct1}, \eqref{17-10-16ct4*} and \eqref{30-11-16ct3*}, it holds the estimate
\begin{align}\label{6-12-16ct2*}
\big\|\overline{u^h_a}(f)\big\|_{\Omega} 
\le \big\|\overline{u^h_a}(f)\big\|_{1,\Omega} 
\le \frac{c_\gamma}{c_1}\big\|\overline{u^h}(f)\big\|_{\Gamma} 
\le \frac{c^2_\gamma}{c_1}\big\|\overline{u^h}(f)\big\|_{1,\Omega} 
\le \frac{c^2_\gamma}{c_1^2}\|f\|_\Omega.
\end{align}
\end{remark}

\begin{remark}
Let 
\begin{align*}
A:= 
\begin{pmatrix}
-\rho\div + u^h_a\\
-\rho\nabla
\end{pmatrix},
\quad \mu^h := \left(
f^h, p^h \right)
\quad \mbox{and} \quad
\nu^h:= \left(
g^h, q^h \right).
\end{align*}
Then the system \eqref{1-12-16ct2*}--\eqref{1-12-16ct2**} can be rewritten in the abbreviation form
\begin{align}\label{6-12-16ct4*}
\left( A(\mu^h), \nu^h - \mu^h\right) \ge 0. 
\end{align}
Furthermore, the optimality of $f^h_{n+1}$ and $p^h_{n+1}$ due to Algorithm \ref{31-8-18ct1} yields that
\begin{equation}\label{5-12-16ct2}
\begin{aligned}
&\left( \frac{1}{\tau} ( f^h_{n+1} -f^h_n) + u^h_a(f^h_n), g^h - f^h_{n+1}\right)_\Omega +\rho \left( p^h_n, \nabla(g^h -f^h_{n+1})\right)_\Omega \ge 0,\\
&\left( \frac{\theta}{\tau} (p^h_{n+1}-p^h_n) -\rho \nabla \widetilde{f}^h_{n+1}, q^h-p^h_{n+1}\right) \ge 0
\end{aligned}
\end{equation}
for all $\nu^h = (g^h,q^h)\in F^h_{ad}\times \mathcal{B}_1\left( (\mathcal{V}^h_0)^d\right)$. 
\end{remark}

We have the following auxiliary result.
\begin{lemma}
Let $\mu^h_{n}:= \left(
f^h_{n}, p^h_{n} \right)$ and  $\mu^h_{n+1}:= \left(
f^h_{n+1}, p^h_{n+1} \right)$. Then the system \eqref{5-12-16ct2} has the form
\begin{align}\label{6-12-16ct4}
\left( A(\mu^h_{n+1}) + B(\mu^h_{n+1}-\mu^h_n), \nu^h - \mu^h_{n+1}\right) \ge 0, 
\end{align}
where the matrix
\begin{align}\label{16-4-18ct2}
B:=
\begin{pmatrix}
\frac{1}{\tau}I - \overline{u^h_a} & \rho\div\\
-\rho\nabla & \frac{\theta}{\tau}I
\end{pmatrix}
\end{align}
is symmetric, positive definite.
\end{lemma}
\begin{proof}
Using the identity \eqref{6-12-16ct2}, we rewrite the system \eqref{5-12-16ct2} in the form
\begin{equation*}
\begin{aligned}
&\left( \frac{1}{\tau} ( f^h_{n+1} -f^h_n) - \overline{u^h_a}( f^h_{n+1} -f^h_n) +\rho\div( p^h_{n+1} -p^h_n) + u^h_a(f^h_{n+1}) -\rho\div p^h_{n+1}, g^h - f^h_{n+1}\right)_\Omega \ge 0,\\
&\left( -\rho \nabla (f^h_{n+1} - f^h_n) + \frac{\theta}{\tau} (p^h_{n+1}-p^h_n) -\rho \nabla f^h_{n+1}, q^h-p^h_{n+1}\right) \ge 0,
\end{aligned}
\end{equation*}
which verifies the inequality \eqref{6-12-16ct4}, where the matrix $B$ by the equation \eqref{6-12-16ct3} is symmetric. Now we show that $B$ is positive definite. For all $\mu^h = (f^h,p^h)\in F^h_{ad}\times \mathcal{B}_1\left( (\mathcal{V}^h_0)^d\right)$, by the inequality \eqref{6-12-16ct2*}, we get
\begin{align*}
\left( \frac{1}{\tau}f^h - \overline{u^h_a}(f^h),f^h\right)_\Omega \ge \frac{1}{\tau}\|f^h\|^2_\Omega - \|\overline{u^h_a}(f^h)\|_\Omega \|f^h\|_\Omega \ge \left( \frac{1}{\tau} - \frac{c^2_\gamma}{c^2_1}\right)\|f^h\|^2_\Omega.
\end{align*}
This together with \eqref{16-4-18ct1} yields that
\begin{align*}
\left(B\mu^h,\mu^h\right)
&= \left( \frac{1}{\tau}f^h - \overline{u^h_a}(f^h),f^h\right)_\Omega - 2\rho \left( \nabla f^h, p^h\right)_\Omega + \frac{\theta}{\tau} \left(p^h,p^h \right)_\Omega \\
&\ge \left( \frac{1}{\tau} - \frac{c^2_\gamma}{c^2_1}\right)\|f^h\|^2_\Omega - 2\rho\|\nabla\|\|f^h\|_\Omega \|p^h\|_\Omega + \frac{\theta}{\tau}\|p^h\|^2_\Omega \\
&= P\left( \|f^h\|_\Omega + \frac{Q}{P} \|p^h\|_\Omega\right)^2 + \frac{PR-Q^2}{P}\|p^h\|^2_\Omega 
\end{align*}
with
\begin{align*}
P:= \frac{1}{\tau} - \frac{c^2_\gamma}{c^2_1}, \quad Q:= - \rho\|\nabla\| \quad\mbox{and}\quad R:= \frac{\theta}{\tau}
\end{align*}
satisfying the relation $PR>Q^2$ due to \eqref{5-12-16ct0}, which finishes the proof.
\end{proof}

\begin{theorem}\label{PDA}
For any fixed $h>0$ let $\left( \mu^h_n\right)_n := \left( f^h_n, p^h_n\right)_n$ be the sequence generated by Algorithm \ref{31-8-18ct1}. 

(i) Then a subsequence $\left( \mu^h_{n_k}\right)_k$ and an element $\mu^h_* := (f^h_*,p^h_*) \in F^h_{ad} \times \partial TV(f^h_*)$ exist such that $\left( \mu^h_{n_k}\right)_k$ converges to $\mu^h_*$ in the finite dimensional space $\mathcal{V}^h_1 \times (\mathcal{V}^h_0)^d$. Furthermore, $\mu^h_*$ satisfies the condition \eqref{6-12-16ct4*}, i.e. $
\left( A(\mu^h_*), \nu^h - \mu^h_*\right) \ge 0 
$
for all $\nu^h\in F^h_{ad}\times \mathcal{B}_1\left( (\mathcal{V}^h_0)^d\right)$ and $f^h_*$ is thus a minimizer of $(\mathcal{P}^h)$.

(ii) The equation
\begin{align}\label{7-12-16ct1}
\|\mu^h_{n+1}-\mu^h_n\|^2 =\mathcal{O}\left( \frac{1}{n}\right) 
\end{align}
holds true, where $\|\cdot\|$ is some norm of the finite dimensional space $\mathcal{V}^h_1 \times (\mathcal{V}^h_0)^d$.
\end{theorem}

\begin{proof}
(i) Denote $\mu^h := (f^h,p^h)$, where $f^h\in F^h_{ad}$ is an arbitrary solution to $(\mathcal{P}^h)$ and $p^h\in \partial TV(f^h)$ generates from $f^h$, due to Lemma \ref{first-order-condition}. Using the notations \eqref{6-12-16ct4*} and \eqref{16-4-18ct2}, we have that
\begin{align}\label{6-12-16ct5}
\|\mu^h_{n+1}-\mu^h\|^2_B 
&= \|\mu^h_n - \mu^h +\mu^h_{n+1}-\mu^h_n\|^2_B \notag\\
&= \|\mu^h_n - \mu^h \|^2_B + \|\mu^h_{n+1}-\mu^h_n\|^2_B - 2\left( B(\mu^h_{n+1}-\mu^h_n), \mu^h-\mu^h_n\right). 
\end{align}
Talking $\nu^h=\mu^h$ in \eqref{6-12-16ct4}, we get
\begin{align*}
\left( B(\mu^h_{n+1}-\mu^h_n), \mu^h-\mu^h_n\right) 
&= \left( B(\mu^h_{n+1}-\mu^h_n), \mu^h-\mu^h_{n+1}\right) + \left( B(\mu^h_{n+1}-\mu^h_n), \mu^h_{n+1}-\mu^h_n\right)\\
&\ge \left( A(\mu^h_{n+1}), \mu^h_{n+1}-\mu^h\right) + \left( B(\mu^h_{n+1}-\mu^h_n), \mu^h_{n+1}-\mu^h_n\right)\\
&= \left( A(\mu^h), \mu^h_{n+1}-\mu^h\right) + \left( A(\mu^h_{n+1}) -A(\mu^h), \mu^h_{n+1}-\mu^h\right) + \|\mu^h_{n+1}-\mu^h_n\|^2_B
\end{align*}
and so that \eqref{6-12-16ct5} yields
\begin{align}\label{6-12-16ct5*}
\|\mu^h_{n+1}-\mu^h\|^2_B 
&\le \|\mu^h_n - \mu^h \|^2_B - \|\mu^h_{n+1}-\mu^h_n\|^2_B \notag\\
&~\quad - 2\left( A(\mu^h), \mu^h_{n+1}-\mu^h\right) - 2\left( A(\mu^h_{n+1}) - A(\mu^h), \mu^h_{n+1}-\mu^h\right).
\end{align}
Talking $\nu^h=\mu^h_{n+1}$ in \eqref{6-12-16ct4*}, we have
\begin{align}\label{6-12-16ct4**}
\left( A(\mu^h), \mu^h_{n+1} - \mu^h\right) \ge 0.
\end{align}
On the other hand, by \eqref{6-12-16ct3}, we get
\begin{align}\label{8-12-16ct1}
&\left( A(\mu^h_{n+1}) - A(\mu^h), \mu^h_{n+1}-\mu^h\right) \notag\\
&~\quad =-\rho\left( \div(p^h_{n+1}-p^h), f^h_{n+1}-f^h\right) + \left(u^h_a(f^h_{n+1}) - u^h_a(f^h), f^h_{n+1}-f^h\right)_\Omega - \rho \left( \nabla(f^h_{n+1}-f^h), p^h_{n+1}-p^h\right)_\Omega\notag\\
&~\quad = \left(u^h_a(f^h_{n+1}) - u^h_a(f^h), f^h_{n+1}-f^h\right)_\Omega\notag\\
&~\quad = \left(\overline{u^h_a}(f^h_{n+1}-f^h), f^h_{n+1}-f^h\right)_\Omega ,\quad \mbox{by}~  \eqref{6-12-16ct2}\notag\\
&~\quad = a\left( \overline{u^h}(f^h_{n+1}-f^h), \overline{u^h_a}(f^h_{n+1}-f^h)\right),\quad \mbox{by}~  \eqref{30-11-16ct3*}\notag\\
&~\quad = \left( \overline{u^h}(f^h_{n+1}-f^h), \overline{u^h}(f^h_{n+1}-f^h)\right)_\Gamma,\quad \mbox{by}~  \eqref{6-12-16ct1}\notag\\
&~\quad \ge 0.
\end{align}
Combining this with \eqref{6-12-16ct4**}--\eqref{6-12-16ct5*}, we arrive at
\begin{align}\label{6-12-16ct5**}
\|\mu^h_{n+1}-\mu^h\|^2_B 
&\le \|\mu^h_n - \mu^h \|^2_B - \|\mu^h_{n+1}-\mu^h_n\|^2_B.
\end{align}
Then for all $n\in \N$ it follows from \eqref{6-12-16ct5**} that
\begin{align}\label{7-12-16ct1*}
\|\mu^h_{n+1}-\mu^h\|^2_B + \sum_{m=0}^n \|\mu^h_{m+1}-\mu^h_m\|^2_B \le \|\mu^h_0 - \mu^h \|^2_B.
\end{align}
Therefore, the sequence ${(\mu^h_n)}_n$ is bounded while the series  $\sum_{n=0}^\infty \|\mu^h_{n+1}-\mu^h_n\|^2_B$ is convergent. A subsequence of ${(\mu^h_n)}_n$ not relabeled and an element $\mu^h_* := (f^h_*, p^h_*)\in F^h_{ad}\times \mathcal{B}_1\left( (\mathcal{V}^h_0)^d\right)$ then exist such that
\begin{align*}
\limn \mu^h_{n+1} = \limn \mu^h_n = \mu^h_*.
\end{align*}
Thus sending $n$ to $\infty$ in \eqref{6-12-16ct4}, we obtain
\begin{align}\label{6-12-16ct4***}
\left( A(\mu^h_*), \nu^h - \mu^h_*\right) \ge 0 
\end{align}
for all $\nu^h\in F^h_{ad}\times \mathcal{B}_1\left( (\mathcal{V}^h_0)^d\right)$. As a direct consequence of \eqref{6-12-16ct4***} (cf.\ Remark \ref{16-4-18ct3}) we get 
$$\left( \nabla f^h_*, q^h - p^h_*\right)_\Omega \le 0 \quad \mbox{for all} \quad q^h\in (\mathcal{V}^h_0)^d, \quad \|q^h\|_\infty \le 1$$
and so that
\begin{align*}
(\nabla f^h_*, p^h_*)_\Omega = \max_{q^h\in (\mathcal{V}^h_0)^d,~\|q^h\|_\infty \le 1} (\nabla f^h_*, q^h)_\Omega = \int_\Omega |\nabla f^h_*| = TV(f^h_*).
\end{align*}
This together with Lemma \ref{subgradient} yields $p^h_* \in \partial TV(f^h_*)$. Consequently, $f^h_*$ is then a minimizer to $(\mathcal{P}^h)$, due to \eqref{6-12-16ct4***} and Lemma \ref{first-order-condition}.

(ii) It remains to show \eqref{7-12-16ct1}. For all $n\in \N$ we get from \eqref{7-12-16ct1*} that
\begin{align}\label{7-12-16ct1**}
\sum_{m=0}^n \|\mu^h_{m+1}-\mu^h_m\|^2_B \le \|\mu^h_0 - \mu^h \|^2_B.
\end{align}
On the other hand, in view of \eqref{6-12-16ct4} we have
\begin{align*}
\left( A(\mu^h_{n+1}) + B(\mu^h_{n+1}-\mu^h_n), \mu^h_{n+2} - \mu^h_{n+1}\right) &\ge 0 \\
\left( A(\mu^h_{n+2}) + B(\mu^h_{n+2}-\mu^h_{n+1}), \mu^h_{n+1} - \mu^h_{n+2}\right) &\ge 0
\end{align*}
which implies
\begin{align*}
\left( A(\mu^h_{n+2}) - A(\mu^h_{n+1}) + B(\mu^h_{n+2}-\mu^h_{n+1}) - B(\mu^h_{n+1}-\mu^h_n), \mu^h_{n+1} - \mu^h_{n+2}\right) \ge 0
\end{align*}
or
\begin{align}\label{16-4-18ct4}
\left( B(\mu^h_{n}-\mu^h_{n+1}) - B(\mu^h_{n+1}-\mu^h_{n+2}), \mu^h_{n+1} - \mu^h_{n+2}\right) \ge \left( A(\mu^h_{n+2}) - A(\mu^h_{n+1}),\mu^h_{n+2} - \mu^h_{n+1}\right).
\end{align}
In view of \eqref{8-12-16ct1} we get
\begin{align}\label{16-4-18ct4*}
\left( A(\mu^h_{n+2}) - A(\mu^h_{n+1}),\mu^h_{n+2} - \mu^h_{n+1}\right)
\ge \left( \overline{u^h}(f^h_{n+2}-f^h_{n+1}), \overline{u^h}(f^h_{n+2}-f^h_{n+1})\right)_\Gamma \ge 0
\end{align}
which together with 
$$\left( B(\mu^h_{n}-\mu^h_{n+1}) - B(\mu^h_{n+1}-\mu^h_{n+2}), (\mu^h_{n}-\mu^h_{n+1}) - (\mu^h_{n+1}-\mu^h_{n+2})\right) = \|\mu^h_{n}-2\mu^h_{n+1}+\mu^h_{n+2}\|^2_B$$
imply that
\begin{align*}
&\|\mu^h_{n}-2\mu^h_{n+1}+\mu^h_{n+2}\|^2_B\\
&= \left( B(\mu^h_{n}-\mu^h_{n+1}) - B(\mu^h_{n+1}-\mu^h_{n+2}), \mu^h_{n}-\mu^h_{n+1}\right) - \left( B(\mu^h_{n}-\mu^h_{n+1}) - B(\mu^h_{n+1}-\mu^h_{n+2}), (\mu^h_{n+1}-\mu^h_{n+2})\right)\\
&\le \left( B(\mu^h_{n}-\mu^h_{n+1}) - B(\mu^h_{n+1}-\mu^h_{n+2}), \mu^h_{n}-\mu^h_{n+1}\right), \quad \mbox{by}~  \eqref{16-4-18ct4}-\eqref{16-4-18ct4*}\\
&= \frac{1}{2}\left( B(\mu^h_{n}-\mu^h_{n+1}) - B(\mu^h_{n+1}-\mu^h_{n+2}), \big((\mu^h_{n}-\mu^h_{n+1}) + (\mu^h_{n+1}-\mu^h_{n+2})\big) + \big((\mu^h_{n}-\mu^h_{n+1}) - (\mu^h_{n+1}-\mu^h_{n+2})\big)\right) \\
&= \frac{1}{2}\left( B(\mu^h_{n}-\mu^h_{n+1}) - B(\mu^h_{n+1}-\mu^h_{n+2}), (\mu^h_{n}-\mu^h_{n+1}) + (\mu^h_{n+1}-\mu^h_{n+2})\right)\\
&~\quad + \frac{1}{2}\left( B(\mu^h_{n}-\mu^h_{n+1}) - B(\mu^h_{n+1}-\mu^h_{n+2}), (\mu^h_{n}-\mu^h_{n+1}) - (\mu^h_{n+1}-\mu^h_{n+2})\right)\\
&= \frac{1}{2}\left( \|\mu^h_{n}-\mu^h_{n+1}\|^2_B - \|\mu^h_{n+1}-\mu^h_{n+2}\|^2_B + \|\mu^h_{n}-2\mu^h_{n+1}+\mu^h_{n+2}\|^2_B\right) 
\end{align*}
and so that
\begin{align*}
\|\mu^h_{n+2}-\mu^h_{n+1}\|^2_B \le \|\mu^h_{n+1}-\mu^h_{n}\|^2_B
\end{align*}
for all $n\in \N$. Therefore, we arrive at
\begin{align*}
(n+1)\|\mu^h_{n+1}-\mu^h_n\|^2_B 
&= \underbrace{\|\mu^h_{n+1}-\mu^h_n\|^2_B+...+\|\mu^h_{n+1}-\mu^h_n\|^2_B}_{(n+1)-\mbox{times}}\\
&\le \|\mu^h_{n+1}-\mu^h_n\|^2_B + \|\mu^h_{n}-\mu^h_{n-1}\|^2_B +...+ \|\mu^h_1-\mu^h_0\|^2_B\\
&= \sum_{m=0}^n \|\mu^h_{m+1}-\mu^h_m\|^2_B \\
&\le \|\mu^h_0 - \mu^h \|^2_B, 
\end{align*}
by \eqref{7-12-16ct1**}. The proof is completed.
\end{proof}

\section{Numerical test}\label{Numer_text}

We now illustrate the theoretical result with numerical examples. For simplicity of exposition we restrict ourselves to the case where $\beta=\sigma=0$ in the system \eqref{17-5-16ct1}, and thus consider the Neumann problem
\begin{align}
-\nabla \cdot \big(\alpha \nabla u \big) &= f^\dag \mbox{~in~} \Omega,  \label{17-5-16ct1*}\\
\alpha \nabla u \cdot \vec{n} &= j  \mbox{~on~} \partial\Omega \label{17-5-16ct3*}
\end{align}
with $\Omega = \{ x = (x_1,x_2) \in {R}^2 ~|~ -1 < x_1, x_2 < 1\}$. Let $\chi_D$ denote the characteristic functional of a Lebesgue measurable set $D\subset \mathbb{R}^2$. 
We assume that entries of the known symmetric diffusion matrix $\alpha$ are defined as
$$\alpha_{11}= 3\chi_{\Omega_{11}} + \chi_{\Omega\setminus\Omega_{11}}, ~
\alpha_{12}= \chi_{\Omega_{12}}, ~
\alpha_{22}= 4\chi_{\Omega_{22}} + 2\chi_{\Omega\setminus\Omega_{22}},$$
where
\begin{align*}
&\Omega_{11} := \left\{ (x_1, x_2) \in \Omega ~\big|~ |x_1| \le 1/2 \mbox{~and~} |x_2| \le 1/2 \right\},~ \Omega_{22} := \left\{ (x_1, x_2) \in \Omega ~\big|~ x_1^2 + x_2^2 \le  1/4 \right\} \mbox{~and~}\\
&\Omega_{12} := \left\{ (x_1, x_2) \in \Omega ~\big|~ |x_1| + |x_2| \le 1/2 \right\},
\end{align*}
see Figure \ref{h1} (bottom row, right). The sought source functional $f^\dag$ in \eqref{17-5-16ct1*} is assumed to be discontinuous and defined as
$$f^\dag = \left( 2 -\frac{\pi}{8}\right)  \chi_{\Omega_{22}} - \frac{\pi}{8} \chi_{\Omega \setminus \Omega_{22}}.$$
We adopt the mesh data structure of \cite{goc06} to our numerical implementation. For this purpose we divide the interval $(-1,1)$ into $\ell$ equal segments, so that the domain $\Omega = (-1,1)^2$ is divided into $2\ell^2$ triangles, where the diameter of each triangle is $h = h_{\ell} = \frac{\sqrt{8}}{\ell}$. 
The Neumann boundary data $j  $ in the equation \eqref{17-5-16ct3*} is chosen as
\begin{equation}\label{20-6-17ct5}
\begin{aligned}
j  
&= \chi_{(0,1]\times\{-1\}} - \chi_{[-1,0]\times\{1\}} + 2\chi_{(0,1]\times\{1\}} -2\chi_{[-1,0]\times\{-1\}}\\
&~\quad +3\chi_{\{-1\}\times(-1,0]} - 3\chi_{\{1\}\times(0,1)} + 4\chi_{\{1\}\times(-1,0]}  - 4\chi_{\{-1\}\times(0,1)} .
\end{aligned}
\end{equation} 
We note that \eqref{17-5-16ct1*}--\eqref{17-5-16ct3*} is the pure Neumann boundary value problem, so $f^\dag$ and $j $ were taken such that the compatibility condition $\int_\Omega f^\dag + \int_{\partial\Omega} j  =0$ is satisfied. Let $u_j(f^\dag)$ denote the unique weak solution of  \eqref{17-5-16ct1*}--\eqref{17-5-16ct3*}. The Dirichlet boundary data on $\Gamma$ of the problem \eqref{17-5-16ct1*}--\eqref{17-5-16ct3*} is computed by 
$$g^\dag := \big(u_j(f^\dag) \big)_{|\Gamma}.$$
We use Algorithm \ref{31-8-18ct1} which is described in Section \ref{primal-dual} for computing the numerical solution of the problem $\left(\mathcal{P}^{h} \right)$. Before doing so, we discuss the constants $c_\gamma$ and $c_1$ appearing in \eqref{5-12-16ct0}. First, we can show the inequality (see, e.g., \cite[\S 7]{HiQu16})
\begin{align}\label{5-9-18ct2}
\frac{\underline{\alpha}}{1 + \left( \sqrt{\frac{3}{2}}\right)^{(d+2)/2}|\Omega|^{1/d}} \|u\|_{1,\Omega} \le a(u,u)
\end{align}
for all $u\in H^1(\Omega)$. Therefore, the constant $c_1$ can be chosen by $\frac{\underline{\alpha}}{1 + \left( \sqrt{\frac{3}{2}}\right)^{(d+2)/2}|\Omega|^{1/d}}$. Furthermore, in the case $\Omega = (-1,1)^2$, the following estimate provides an upper bound for $c_\gamma$.

\begin{lemma}
If $\Omega := (a_1,b_1) \times (a_2,b_2) \times\ldots \times(a_d,b_d) \ni 0$, then $c_\gamma$ can be chosen by $\sqrt{\frac{d+|\overline{\hbar}|^2}{|\underline{\hbar}|}}$, i.e. for all $u\in H^1(\Omega)$, the estimate
\begin{align}\label{5-9-18ct1}
\|\gamma u\|_{L^2(\partial\Omega)} \le \sqrt{\frac{d+|\overline{\hbar}|^2}{|\underline{\hbar}|}} \|u\|_{H^1(\Omega)}
\end{align}
holds true, where $\underline{\hbar} = \min \{|a_1|, |b_1|, \ldots, |a_d|, |b_d|\}$ and $\overline{\hbar} = \max \{|a_1|, |b_1|, \ldots, |a_d|, |b_d|\}$.
\end{lemma}

\begin{proof}
Denote by $\underline{S}_d := \left\{\big(x', \underline{\varphi}(x')\big) \in \overline{\Omega} ~\big|~ x' := (x_1,\ldots,x_{d-1}) ~\mbox{and}~ \underline{\varphi}(x') \equiv a_d\right\}$. First, we assume that $u\in C^1(\overline{\Omega})$. Then, we get
\begin{align*}
u(x',a_d) = -\frac{1}{a_d}\int_{a_d}^0 \frac{\partial\big(x_du(x',x_d)\big)}{\partial x_d}dx_d = -\frac{1}{a_d} \int_{a_d}^0 \left( u(x',x_d) + x_d\frac{\partial u(x',x_d)}{\partial x_d} \right)dx_d
\end{align*}
which implies that
\begin{align*}
|a_d|^2 |u(x',a_d)|^2 
&\le \left(\int_{a_d}^0 \left| u(x',x_d) + x_d\frac{\partial u(x',x_d)}{\partial x_d} \right|dx_d\right)^2 \\
&\le |a_d| \int_{a_d}^0 \left| \sqrt{d} \frac{u(x',x_d)}{\sqrt{d}} + x_d\frac{\partial u(x',x_d)}{\partial x_d} \right|^2dx_d \\
&\le |a_d| \big(d+|a_d|^2\big) \int_{a_d}^0 \left( \frac{u^2(x',x_d)}{d} + \left(\frac{\partial u(x',x_d)}{\partial x_d} \right)^2\right)dx_d.
\end{align*}
Multiplying this inequality by $\sqrt{1+ \left(\frac{\partial\underline{\varphi}}{x_1}\right)^2 +\ldots+ \left(\frac{\partial\underline{\varphi}}{x_{d-1}}\right)^2} \equiv 1$,  integrating over $Q:= (a_1,b_1) \times (a_2,b_2) \times\ldots \times(a_{d-1},b_{d-1})$, we obtain
\begin{align*}
\|u\|^2_{L^2(\underline{S}_d)} 
&\le \frac{d+|\overline{\hbar}|^2}{|\underline{\hbar}|} \int_Q \left(\int_{a_d}^0 \left( \frac{u^2(x',x_d)}{d} + \left(\frac{\partial u(x',x_d)}{\partial x_d} \right)^2\right)dx_d\right)dx'.
\end{align*}
Likewise, denoting by $\overline{S}_d := \left\{\big(x',\overline{\varphi}(x')\big) \in \overline{\Omega} ~\big|~ x' := (x_1,\ldots,x_{d-1}) ~\mbox{and}~ \overline{\varphi}(x') \equiv b_d\right\}$, we also get
\begin{align*}
\|u\|^2_{L^2(\overline{S}_d)} 
&\le \frac{d+|\overline{\hbar}|^2}{|\underline{\hbar}|} \int_Q \left(\int_0^{b_d} \left( \frac{u^2(x',x_d)}{d} + \left(\frac{\partial u(x',x_d)}{\partial x_d} \right)^2\right)dx_d\right)dx'
\end{align*}
and arrive at
\begin{align*}
\|u\|^2_{L^2(\underline{S}_d \cup \overline{S}_d)} 
\le \frac{d+|\overline{\hbar}|^2}{|\underline{\hbar}|} \int_\Omega \left( \frac{u^2(x)}{d} + \left(\frac{\partial u(x)}{\partial x_d} \right)^2\right)dx.
\end{align*}
This implies that \eqref{5-9-18ct1} is satisfied for all $u\in C^1(\overline{\Omega})$. By the everywhere dense property of the set $C^1(\overline{\Omega})$ in $H^1(\Omega)$, we conclude that \eqref{5-9-18ct1} is also fulfilled for all $u\in H^1(\Omega)$, which finishes the proof.
\end{proof}

The constant $\underline{\alpha}$ in \eqref{27-5-19ct1} is taken by $0.1$. With $d=2$, $\Omega = (-1,1)^2$ we then can take $c_1= 0.025$ and $c_\gamma = \sqrt{3}$. As an initial approximation we choose $f^{h_\ell}_0(x) \equiv 1 \in \mathcal{V}^{h_\ell}_1$ and $p^{h_\ell}_0(x) \equiv (\frac{1}{2},\ldots,\frac{1}{2}) \in (\mathcal{V}^{h_\ell}_0)^d$. In the computation below we choose $\tau = 2\cdot10^{-4}$ and $\theta = 5\cdot10^{-2}$.

For observations with noise we assume that
\begin{align}\label{3-7-17ct1}
z_{\delta_{\ell}} = g^\dag+ \theta_\ell\cdot R_{g^\dag} \quad \mbox{for some} \quad \theta_\ell>0 \quad \mbox{depending on} \quad \ell,
\end{align}
where $R_{g^\dag}$ is the matrix of random numbers on the interval $(-1,1)$. We here refer to the discussion in \cite{cla12} which states that the $L^\infty$-misfit term is well suited for this uniformly distributed noise model. However, since $L^\infty$ is a non-reflexive Banach space and the $L^\infty$-norm is not differentiable the use of this misfit term would cause difficulties in the numerical treatment of ill-posed inverse problems.

The measurement error is then computed as $\delta_\ell = \big\|z_{\delta_\ell} -g^\dag\big\|_{L^2(\Gamma)}.$ 
To satisfy the condition \eqref{31-8-16ct1} in Theorem \ref{stability2}  we take $h=h_\ell$, $\rho = \rho_\ell = 10^{-3}h^{1/2}_\ell$ and $\theta_\ell = h_\ell \rho^{1/2}_\ell$. 

Starting with the coarsest level $\ell =4$, we follow \cite[pp.\ 16]{kel99} in the use of the stopping criterion. Here at each iteration $k$ we compute
\begin{align}\label{28-5-19ct1}
\mbox{Tolerance}:= \big\|\nabla J^{h_{\ell}} \big( f^{h_\ell}_k \big) \big\|_{L^2(\Omega)}-\tau_1-\tau_2\big\|\nabla J^{h_{\ell}} \big( f^{h_\ell}_0 \big) \big\|_{L^2(\Omega)},
\end{align}
with $\tau_1 := 10^{-5}h_\ell^{1/2}$ and $\tau_2 := 10^{-4}h_\ell^{1/2}$ and the iteration was stopped if
$\mbox{Tolerance} \le 0$
or the number of iterations reached the maximum iteration count of $N=600$. 

After obtaining the numerical solution of the first iteration process with respect to the coarsest level $\ell=4$, we use its interpolation on the next finer mesh $\ell=8$ as an initial approximation $(f^{h_\ell}_0,p^{h_\ell}_0)$ for the algorithm on this finer mesh, and proceed similarly on the refinement levels.

Let $f_\ell$ denote the functional which is obtained at {\it the final iterate} of Algorithm \ref{31-8-18ct1} corresponding to the refinement level $\ell$. Furthermore, we denote by $u_\ell$ {\it the computed numerical solution} to the Dirichlet problem
\begin{align*}
-\nabla \cdot (\alpha\nabla u) = f_\ell \quad\mbox{in}\quad \Omega \quad\mbox{and}\quad
u =
\begin{cases}
\big(u_j(f^\dag) \big)_{|\partial\Omega\setminus\Gamma} &\mbox{on}\quad \partial\Omega\setminus\Gamma,\\
\big(u_j(f_\ell) \big)_{|\Gamma} &\mbox{on}\quad \Gamma,
\end{cases}
\end{align*}
while by $u^\dag$ the solution of the equation \eqref{17-5-16ct1*} supplemented with the Dirichlet boundary condition $\big(u_j(f^\dag) \big)_{|\partial\Omega}$.

First, we consider the case where $\Gamma$ is the bottom surface of the domain $\Omega$, i.e. $\Gamma = \{ x = (x_1,x_2) \in \overline{\Omega} ~|~ x_2=-1\}$. The numerical results of this case are summarized in Table \ref{b1} and Table \ref{b2}, where we present the refinement level $\ell$, mesh size $h_\ell$ of the triangulation, regularization parameter $\rho_\ell$, measured noise $\delta_\ell$, number of iterations, value of tolerances and errors $\big\|f^\dag - f_\ell \big\|_{\Omega}$, $\big\|u^\dag - u_\ell\big\|_{\Omega}$ and $\big\|u^\dag - u_\ell\big\|_{1,\Omega}$. 

\begin{table}[H]
\begin{center}
\begin{tabular}{|c|l|l|l|l|l|}
\hline
$\ell$ &\scriptsize $h_\ell$ &\scriptsize $\rho_\ell$ &\scriptsize $\delta_\ell$ &\scriptsize {Iterate} &\scriptsize {Tolerance} \\
\hline
4   &0.7071 & 8.4090e-4& 2.3763e-2& 152&  -4.0518e-6\\
\hline
8  &0.3536 & 5.9460e-4& 8.8717e-3& 264& -2.8826e-7\\
\hline
16  &0.1766 & 4.2045e-4& 2.5872e-3& 343&   -1.4460e-8\\
\hline
32  &8.8388e-2 & 2.9730e-4& 1.1817e-3& 471&  -1.1382e-7\\
\hline
64  &4.4194e-2 & 2.1022e-4& 5.6112e-4& 561& -4.6128e-8\\
\hline
\end{tabular}
\caption{Refinement level $\ell$, mesh size $h_\ell$ of the triangulation, regularization parameter $\rho_\ell$,  measurement noise $\delta_\ell$, number of iterates and value of tolerances.}
\label{b1}
\end{center}
\end{table}

\begin{table}[H]
\begin{center}
\begin{tabular}{|c|l|l|l|}
 \hline
$\ell$ &\scriptsize $\big\|f^\dag - f_\ell \big\|_{\Omega}$ &\scriptsize $\big\|u^\dag - u_\ell\big\|_{\Omega}$ &\scriptsize $\big\|u^\dag - u_\ell\big\|_{1,\Omega}$ \\
\hline
4&    0.6829     & 1.5022e-2    & 3.9121e-2 \\
\hline
8&   0.3506     & 8.5904e-3    & 3.3448e-2 \\
\hline
16&   0.1692   & 3.2470e-3    & 2.2919e-2 \\
\hline
32&   0.1095   & 1.2926e-3    & 1.5372e-2 \\
\hline
64&   0.7403e-2  & 6.0377e-4    & 1.2252e-2 \\
\hline
\end{tabular}
\caption{Errors  $\big\|f^\dag - f_\ell \big\|_{\Omega}$,  $\big\|u^\dag - u_\ell\big\|_{\Omega}$ and $\big\|u^\dag - u_\ell\big\|_{1,\Omega}$.}
\label{b2}
\end{center}
\end{table}

\begin{figure}[H]
\begin{center}
\includegraphics[scale=0.4]{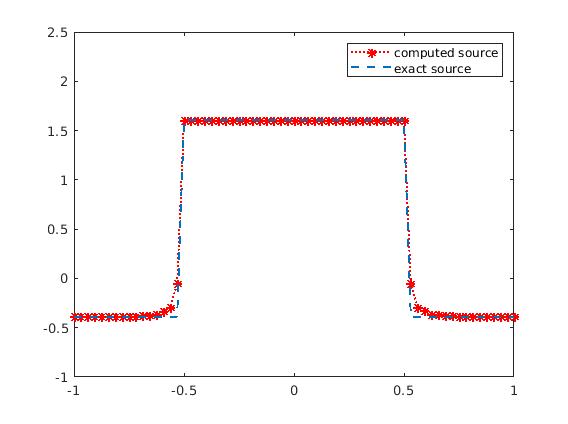}
\includegraphics[scale=0.4]{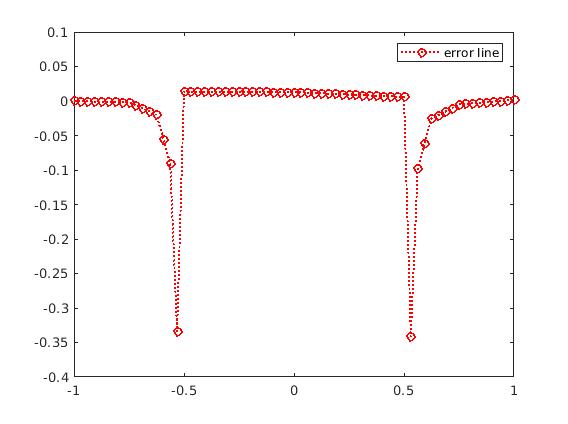}
\includegraphics[scale=0.4]{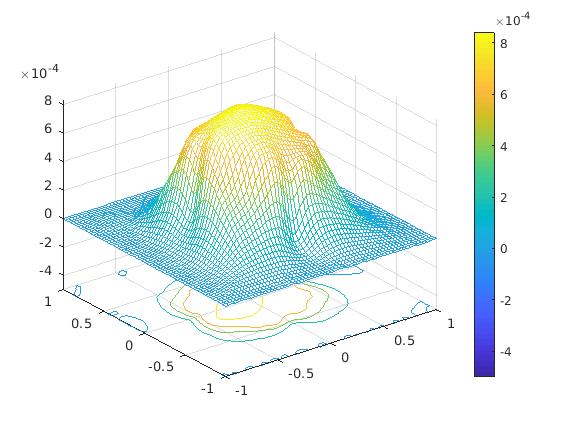}
\includegraphics[scale=0.4]{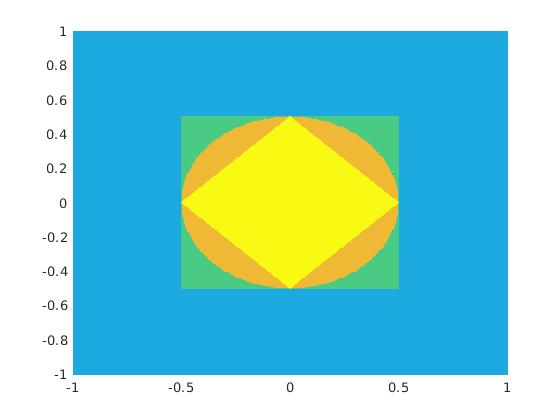}
\end{center}
\caption{Observations taking at the bottom surface $(\ell = 64)$: Interpolation $I^{h_\ell}_1 f^\dag(x_1,0)$ together with computed numerical solution $f_\ell(x_1,0)$ of the algorithm at the final iteration (Top row, left), the differences $I^{h_\ell}_1 f^\dag(x_1,0) - f_\ell(x_1,0)$ (Top row, right) and $u^\dag - u_\ell$ (Bottom row, left). Bottom row, right: the domain $\Omega$ together with the subdomains $\Omega_{11}$, $\Omega_{22}$, and $\Omega_{12}$.}
\label{h1}
\end{figure}

We now consider the case where $\Gamma$ includes the bottom surface and the left surface of the domain $\Omega$, i.e. $\Gamma = \{ x = (x_1,x_2) \in \overline{\Omega} ~|~ x_2=-1\} \cup \{ x = (x_1,x_2) \in \overline{\Omega} ~|~ x_1=-1\}$. To try obtaining an error level similar to the case where observations were taken on the bottom surface only, $\theta_\ell$ in \eqref{3-7-17ct1} is changed to $\theta_\ell = \frac{1}{2}h_\ell \rho^{1/2}_\ell$. 

Computations show that the measurement noise $\delta_\ell = 4.6039.10^{-4}$, $\mbox{Tolerance} = -1.8493.10^{-8}$, the iteration stops at $n=477$. Errors $\big\|f^\dag - f_\ell\big\|_{\Omega} = 6.1605.10^{-2}$,  $\big\|u^\dag - u_\ell\big\|_{\Omega}=5.8110.10^{-4}$ and $\big\|u^\dag - u_\ell\big\|_{1,\Omega}=1.1164.10^{-2}$. 

\begin{figure}[H]
\begin{center}
\includegraphics[scale=0.4]{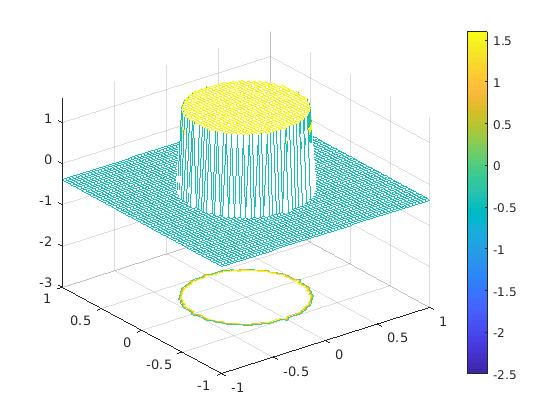}
\includegraphics[scale=0.4]{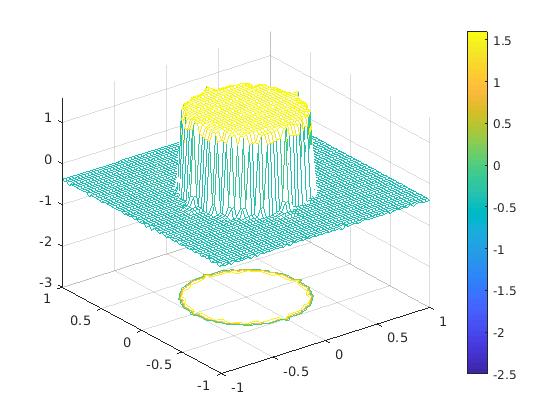}
\includegraphics[scale=0.4]{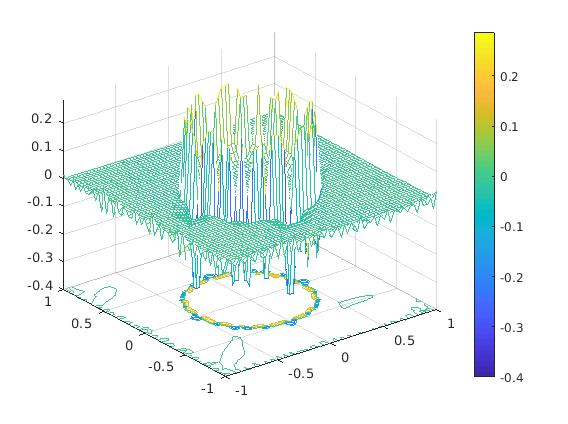}
\includegraphics[scale=0.4]{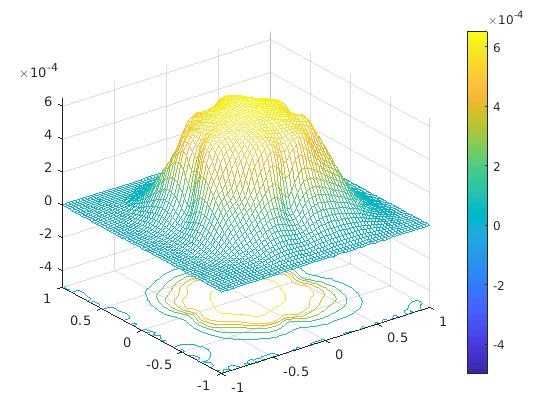}
\end{center}
\caption{Observations taking at the bottom surface and the left surface $(\ell = 64)$: Interpolation $I^{h_\ell}_1 f^\dag$ (Top row, left), computed numerical solution $f_\ell$ of the algorithm at the final  iteration (Top row, right), the differences $I_1^{h_{\ell}} f^\dag - f_\ell$ (Bottom row, left), $u^\dag - u_\ell$ (Bottom row, right).}
\label{h2}
\end{figure}

\section*{Acknowledgements} 
The authors M. Hinze and T.N.T. Quyen would like to thank the
Referee and the Editor for their valuable comments and suggestions which helped to improve our paper.

\end{document}